\documentclass[12pt]{amsart}
\usepackage{amssymb,amsmath,amsthm,bm,mathrsfs}
\usepackage[utf8]{inputenc}
\usepackage[T1]{fontenc}
\usepackage[colorinlistoftodos,prependcaption,textsize=tiny]{todonotes}
 \definecolor{darkblue}{RGB}{0,0,160}
\usepackage{fouriernc} 
\usepackage[colorlinks=true,allcolors=darkblue]{hyperref}
\DeclareSymbolFont{usualmathcal}{OMS}{cmsy}{m}{n}
\DeclareSymbolFontAlphabet{\mathcal}{usualmathcal}

\usepackage[T1]{fontenc}

\usepackage{color}

\def\d{{\rm d}}

\def\G{{\mathcal G}}

\def \and{\qquad\text{and}\qquad}

\newcommand{\supp}{\mathrm{supp}\,}

\newcounter{thms}
\newcounter{other}
\numberwithin{other}{section}

\newtheorem{theorem}[thms]{Theorem}
\newtheorem*{theorem*}{Theorem}
\newtheorem*{proposition*}{Proposition}

\numberwithin{corollary}{thms}
\newtheorem{lemma}[other]{Lemma}
\newtheorem{definition}[other]{Definition}
\newtheorem{remark}[other]{Remark}
\newtheorem{conjecture}{Conjecture}

\def\vint_#1{\mathchoice%
      {\mathop{\kern 0.2em\vrule width 0.6em height 0.69678ex depth -0.58065ex
              \kern -0.8em \intop}\nolimits_{\kern -0.4em#1}}%
      {\mathop{\kern 0.1em\vrule width 0.5em height 0.69678ex depth -0.60387ex
              \kern -0.6em \intop}\nolimits_{#1}}%
      {\mathop{\kern 0.1em\vrule width 0.5em height 0.69678ex depth -0.60387ex
              \kern -0.6em \intop}\nolimits_{#1}}%
      {\mathop{\kern 0.1em\vrule width 0.5em height 0.69678ex depth -0.60387ex
              \kern -0.6em \intop}\nolimits_{#1}}}
\def\vintslides_#1{\mathchoice%
      {\mathop{\kern 0.1em\vrule width 0.5em height 0.697ex depth -0.581ex
              \kern -0.6em \intop}\nolimits_{\kern -0.4em#1}}%
      {\mathop{\kern 0.1em\vrule width 0.3em height 0.697ex depth -0.604ex
              \kern -0.4em \intop}\nolimits_{#1}}%
      {\mathop{\kern 0.1em\vrule width 0.3em height 0.697ex depth -0.604ex
              \kern -0.4em \intop}\nolimits_{#1}}%
      {\mathop{\kern 0.1em\vrule width 0.3em height 0.697ex depth -0.604ex
              \kern -0.4em \intop}\nolimits_{#1}}}

\newcommand{\aveint}[2]{\mathchoice%
      {\mathop{\kern 0.2em\vrule width 0.6em height 0.69678ex depth -0.58065ex
              \kern -0.8em \intop}\nolimits_{\kern -0.45em#1}^{#2}}%
      {\mathop{\kern 0.1em\vrule width 0.5em height 0.69678ex depth -0.60387ex
              \kern -0.6em \intop}\nolimits_{#1}^{#2}}%
      {\mathop{\kern 0.1em\vrule width 0.5em height 0.69678ex depth -0.60387ex
              \kern -0.6em \intop}\nolimits_{#1}^{#2}}%
      {\mathop{\kern 0.1em\vrule width 0.5em height 0.69678ex depth -0.60387ex
              \kern -0.6em \intop}\nolimits_{#1}^{#2}}}

\renewcommand*{\cdots}{%
  \mathinner{{\cdotp}{\cdotp}{\cdotp}}%
}

\numberwithin{equation}{section}

\title[A cone restriction estimate using polynomial partitioning]{A cone restriction estimate using polynomial partitioning}

 \author[Yumeng Ou]{Yumeng Ou}
 \address{\noindent Department of Mathematics, University of Pennsylvania, Philadelphia, PA 19104, USA }
\email{yumengou@sas.upenn.edu} 
 \author[Hong Wang]{Hong Wang}
 \address{\noindent School of Mathematics, Institute for Advanced Study, Princeton, NJ 08540, USA }
\email{hongwang@ias.edu} 


\begin{document}
\begin{abstract} 
We obtain improved Fourier restriction estimate for the truncated cone using the method of polynomial partitioning in dimension $n\geq 3$, which in particular solves the cone restriction conjecture for $n=5$, and recovers the sharp range for $3\leq n\leq 4$. The main ingredient of the proof is a $k$-broad estimate for the cone extension operator, which is a weak version of the $k$-linear cone restriction conjecture for $2\leq k\leq n$.
\end{abstract}
\maketitle

\section{Introduction and main results}

In this article, we obtain an improved Fourier restriction estimate for the cone in all dimensions $n\geq 3$, and in particular solve the cone restriction conjecture of Stein \cite{Stein} in dimension $n=5$. 

The  \emph{Fourier restriction conjecture} is  one of the central open problems in harmonic analysis. 
It concerns a very basic question: whether $\hat{f}$, the Fourier transform of a function $f$, can be meaningfully restricted onto a hypersurface. 
Stein \cite{Stein} conjectured that for well curved surfaces such as sphere, paraboloid, or the cone studied in the present article, this is indeed the case. The precise statement of Stein's conjecture  about the cone is the following. 
\begin{conjecture}[Cone restriction conjecture]\label{conj}
Let $n\geq 3$ and $\mathcal{C}$ be the truncated cone in $\mathbb{R}^n$ defined in (\ref{def:cone}) below. For all $1\leq p<\frac{2(n-1)}{n}$, there holds
\[
\left\|\hat{f}|_{\mathcal{C}}\right\|_{L^p(\mathcal{C};\, d\sigma)}\leq C_{p} \|f\|_{L^p(\mathbb{R}^n)},\quad\quad \forall f\in L^p(\mathbb{R}^n),
\]where $d\sigma$ denotes the surface measure on $\mathcal{C}$.
\end{conjecture}

This conjecture was so far known to be true only in dimensions $n=3$ and $4$, proved by Barcelo \cite{Barcelo} and Wolff \cite{Wolff} respectively. The main contribution of the present article is the resolution of the conjecture in dimension $n=5$ and improved partial results towards higher dimensional cases. 

The Fourier restriction conjecture on various surfaces with enough curvature is directly connected to many open conjectures in analysis and PDE including the Kakeya conjecture, the Bochner--Riesz conjecture, and the local smoothing conjecture for wave equations. It is also known that certain versions of the restriction estimates can be used to study problems in other related fields such as geometric measure theory (e.g. the Falconer's distance set conjecture) and analytic number theory (e.g. estimating the number of solutions to Diophantine equations). It has been extensively studied for decades, and we refer to \cite{BG, Guth1, Guth2, BD} and the references therein for historical remarks on the problem and the aforementioned connections. However, there are very few surfaces and dimensions for which a sharp restriction theorem is known. For example, the restriction conjecture for the paraboloid and the sphere remains open for $n\geq 3$. Moreover, it is known (\cite{Tao2}) that there is a certain link between the restriction estimate for the cone in $\mathbb{R}^{n+1}$ and that for the paraboloid, sphere, or other conic sections in $\mathbb{R}^n$, which suggests possible further applications of our result.

We now describe the precise formulation of the main theorem. Let $B^{n-1}$ be the unit ball in $\mathbb{R}^{n-1}$ and denote its closure by $\bar{B}^{n-1}$. Given a function $f: 2\bar{B}^{n-1}\setminus B^{n-1}\to \mathbb{C}$, where $2\bar{B}^{n-1}\setminus B^{n-1}$ denotes the closed annulus $\{\xi\in \mathbb{R}^{n-1}:\,1\leq |\xi|\leq 2\}$, define the \emph{truncated cone} as
\begin{equation}\label{def:cone}
\mathcal{C}=\left\{(\xi,\xi_n)\in\mathbb{R}^{n-1}\times\mathbb{R}:\,\xi_1^2+\cdots+\xi_{n-1}^2=\xi_n^2,\,\, 1\leq \xi_n \leq 2 \right\}.
\end{equation}Then the associated Fourier extension operator from the truncated cone $\mathcal{C}$ is
\[
Ef(x):=\int_{2\bar{B}^{n-1}\setminus B^{n-1}} e^{i(x_1\xi_1+\cdots+x_{n-1}\xi_{n-1}+x_n|\xi|)} f(\xi)\,\d\xi.
\]

It is well known that a Fourier restriction estimate is equivalent to a Fourier extension estimate by a short duality argument. Therefore, our main restriction estimate can be formulated as the following:

\begin{theorem}\label{main}
For $n\geq 3$, the extension operator $E$ from the cone satisfies
\[
\|Ef\|_{L^p(\mathbb{R}^n)}\leq C_p \|f\|_{L^p(2\bar{B}^{n-1}\setminus B^{n-1})},\quad \forall f\in L^p(2\bar{B}^{n-1}\setminus B^{n-1})
\]whenever
\begin{equation}\label{p range}
p>
\begin{cases}
4 &\text{if}\,\, n=3,\\
2\cdot \frac{3n+1}{3n-3} &\text{if}\,\, n>3\,\,\text{odd},\\
2\cdot \frac{3n}{3n-4}&\text{if}\,\, n>3\,\,\text{even}.
\end{cases}
\end{equation}
\end{theorem}

When $n=3$ and $n=4$, this recovers the sharp range of $p$ ($p>4$ for $n=3$ by Barcelo \cite{Barcelo} and $p>3$ for $n=4$ by Wolff \cite{Wolff}) for which the cone restriction estimate holds true. When $n=5$, Theorem \ref{main} derives for the first time the sharp range $p>\frac{8}{3}$ for the cone restriction estimate. When $n\geq 6$, Theorem \ref{main} provides new partial progress towards the sharp range $p>\frac{2(n-1)}{n-2}$ corresponding to Conjecture \ref{conj}. Before our work, the best known range for $n>4$ was $p>\frac{2(n+2)}{n}$, proved by Wolff \cite{Wolff} using a bilinear method.

It is also interesting to study $L^q\to L^p$ restriction estimate for $q\neq p$. Stein conjectured that $E: L^q\to L^p$ whenever $p>\frac{2(n-1)}{n-2}$ and $q'\leq \frac{n-2}{n}p$ based on the Knapp example. When $q>p$, such estimate is immediately implied by Theorem \ref{main} using  H\"older's inequality.  When $q<p$, one can obtain the following estimate by slightly modifying the proof of Theorem \ref{main}.
\begin{theorem}\label{mainpq}
For $n\geq 3$, the operator $E$ defined above satisfies the estimate
\[
\|Ef\|_{L^p(\mathbb{R}^n)}\leq C_{p,q} \|f\|_{L^q(2\bar{B}^{n-1}\setminus B^{n-1})},\quad \forall f\in L^q(2\bar{B}^{n-1}\setminus B^{n-1})
\]whenever the tuple $(p,q,k)$ is admissible in the sense that $q>2$, $2\leq k\leq n$ and
\begin{equation}\label{admpq}
\begin{cases}
p> 2\cdot\frac{n+2}{n},\,q'\leq \frac{n-2}{n}p, &\text{if}\,\, k=2,\\
p> 2\cdot\frac{n+k}{n+k-2},\, p\geq\frac{n}{\frac{2n-k-1}{2}-\frac{n-k+1}{q}}, &\text{if}\,\,k\geq 3.
\end{cases}
\end{equation}

\end{theorem} 
For each fixed $n$, one can optimize the range of $L^q\to L^p$ restriction estimate above by choosing the most suitable $k$. In particular, in the case $n=5$, taking $k=3$, Theorem \ref{mainpq} implies the optimal conjectured range $p>\frac{8}{3}, \,q'\leq \frac{3}{5}p$. The result in the open range of (\ref{admpq}) follows from a similar argument for Theorem \ref{main}. In order to obtain the end point estimate, we apply the bilinear interpolation with a bilinear cone restriction estimate obtained in \cite{Wolff}. The interpolation argument is adapted from the work of Tao, Vargas and Vega \cite{TVV}, where the  paraboloid version of the question is studied. 

We prove the theorems above using polynomial partitioning. 
 The idea of applying polynomial method in harmonic analysis dates back to the resolution of the finite field Kakeya problem by Dvir \cite{D}. Later on, Guth and Katz  introduced in \cite{GK}  polynomial partitioning techniques to solve the Erd\"os distinct distances problem in combinatorics. In 2014, Guth \cite{Guth1, Guth2} introduced polynomial partitioning into the study of restriction estimates (also see Wang \cite{Wang} and Hickman--Rogers \cite{HR} for further refinements), which was later used by Du, Guth and Li \cite{DGL} to solve the Schr\"{o}dinger maximal estimate in $\mathbb{R}^2$.

More precisely, polynomial partitioning will be used in Section \ref{SecBroad} where we prove a $k$-broad restriction inequality on the cone (Theorem \ref{kBroadThm}), which is a weak version of the $k$-linear restriction estimate (\ref{klinear}). The $k$--broad to linear reduction (i.e. how Theorem \ref{kBroadThm} implies Theorem \ref{main}) is similar to the arguments in \cite{Guth2} for the paraboloid: it will be obtained by the $k$-broad estimates together with decoupling and the Lorentz rescaling. 

Compared to the case of the paraboloid treated in \cite{Guth1, Guth2}, the main novelty of our argument is as follows. In the paraboloid case, polynomial partitioning reduces the problem to a lower dimensional problem using the so-called \emph{equidistribution property}. This property, however, fails to hold  in the case of the cone. The failure essentially boils down to the fact that the cone has vanishing curvature in one direction at each point, thus the geometry of the resulting wave packets is more subtle. To overcome this geometric obstruction, we remove a negligible part of $Ef$ and show that the equidistribution property holds for the remaining part.  We leave more detailed discussion in this regard to Section \ref{sec: equi}. In addition, an important ingredient in the induction by scales argument is to understand how the wave packet decomposition (see Section \ref{sec: pre} for definition) at various scales are related to each other. This is another step in the proof where the cone has to be treated very differently from the paraboloid. We address this question in Section \ref{subsec: largesmall}, which seems to be of its own interest. 
 
The article is planned as follows. In Section \ref{sec: pre}, we recall several common notations and basic tools in restriction theory (e.g. wave packet decomposition). Then, in Section \ref{SecBroad} we introduce the aforementioned $k$-broad restriction inequality (Theorem \ref{kBroadThm}), which will be applied to obtain the main results Theorem \ref{main} and \ref{mainpq} in Section \ref{linearmain}. The proof of Theorem \ref{kBroadThm} is provided in Section \ref{sec: proof} using polynomial partitioning, before which the basic setup of the polynomial partitioning method is introduced in Section \ref{strategy}.

\section{Preliminaries}\label{sec: pre}
\subsection{Notations.}
Throughout the paper, we work with smooth functions $f,g:\, 2\bar{B}^{n-1}\setminus B^{n-1}\to \mathbb{C}$. We use $B_R^n$ to denote an arbitrary ball in $\mathbb{R}^n$ of radius $R$ and oftentimes write $B_R=B_R^n$ for short when the dimension of the space is clear from the context. The $\alpha$-neighborhood of a set $E$ is denoted by $N_\alpha(E)$. Let $Z$ be an algebraic variety in $\mathbb{R}^n$, then its tangent plane at $z\in Z$ is denoted by $T_{z}Z$.

Our arguments will involve frequently a small parameter $\epsilon>0$ and a large parameter $R>1$. Given positive numbers $A,B$ and a list of quantities $L$, we use $A\lesssim_L B$ to denote $A\leq C_L B$ for some absolute constant depending only on $L$ and possibly the dimension $n$. Similarly for $A\gtrsim_L B$. $A\sim_L B$ is used if both $A\lesssim_L B$ and $A\gtrsim_L B$ hold. And $O_L(1)$ denotes a quantity that is smaller than a constant depending on $L$ only. Moreover, $A\lessapprox B$ if $A\leq C_\epsilon R^\epsilon B$ for any $\epsilon>0$, $R>1$. 

We say a quantity is ${\rm RapDec}(R)$ if it is bounded by a huge negative power of $R$, which makes it negligible in our arguments. A function $Ef$ is said to be \emph{essentially supported} in a set $\Omega$ with an underlying parameter $R$ if all appropriate norms concerned in the problem of the tail of $Ef$ outside $\Omega$ is ${\rm RapDec}(R)$. 

\subsection{Wave packet decomposition}
We briefly recall the wave packet decomposition,  an essential tool in our argument. Fix a large parameter $R\gg 1$. Cover the region $2\bar{B}^{n-1}\setminus B^{n-1}$  by finitely overlapping sectors $\theta$ of length $1$ (in the radial direction) and angular width $R^{-1/2}$. Let $\{\psi_{\theta}\}$ be a smooth partition of unity subordinate to this cover, and write $f=\sum_{\theta} \psi_{\theta}f$. 

Next, we break up $\psi_{\theta}f$ according to frequency.  We do it in two steps. First, we cover $\mathbb{R}^{n-1}$ by finitely overlapping cubes of side length $\sim R^{\frac{1}{2}+\delta}$, centered at $v\in R^{\frac{1}{2}+\delta} \mathbb{Z}^{n-1}$. Here $\delta>0$ is a fixed small parameter. Let $\{\eta_{v}\}$ be a smooth partition of unity subordinate to this cover. We write 
$$ f=\sum_{\theta,\,v} (\eta_v (\psi_{\theta} f)^{\wedge} )^{\vee}.$$

Second, we break the function $ (\eta_v (\psi_{\theta} f)^{\wedge} )^{\vee}$ into even finer pieces according to $\theta$. More precisely, for each $\theta$, let $\xi_{\theta}=(\xi_{\theta,1}, \dots, \xi_{\theta,n-1})$ be the point on the central line of $\theta$ with $|\xi_{\theta}|=1$. We cover the $R^{1/2+\delta}$--cube centered at $v$ by parallel thin plates $P^{\ell}_{\theta,\,v}$ of radius $\sim R^{\frac{1}{2}+\delta}$ and thickness $R^{\frac{\delta}{2}}$, where the normal direction of $P^{\ell}_{\theta,v}$ is $\xi_{\theta}$ and $\ell = 1, \dots, \sim R^{\frac{1}{2}}$. Let $\left\{\eta^{\ell}_{\theta,v}\right\}$ be a smooth partition of unity subordinate to this cover.  We write 
$$f= \sum_{\theta,\,v, \,\ell} (\eta^{\ell}_{v, \theta}\eta_v (\psi_{\theta} f)^{\wedge} )^{\vee}.$$
Note that $(\eta_v)^{\wedge}(\xi)$ is rapidly decaying for  $\xi$ outside of $\left\{\xi\in \mathbb{R}^{n-1}:\, |\xi|\gtrsim R^{-\frac{1}{2}} \right\}$ and $(\eta^{\ell}_{\theta,v})^{\wedge}(\xi)$ is rapidly decaying for $\xi\in \mathbb{R}^{n-1}$ outside a thin tube of length $1$ and radius $R^{-\frac{1}{2}}$ pointing in direction $\xi_\theta$. We can choose smooth functions $\widetilde{\psi}_{\theta}$  so that $\widetilde{\psi}_{\theta}$ is essentially  supported on $\theta$, and $\widetilde{\psi}_{\theta}=1$ on a $cR^{-1/2}$ neighborhood of the support of $\psi_{\theta}$ for a small constant $c>0$. Now we define 
$$f^{\ell}_{\theta,v} := \widetilde{\psi}_{\theta} [ (\eta^{\ell}_{v, \theta}\eta_v (\psi_{\theta} f)^{\wedge} )^{\vee}].$$

Because of the rapid decay of $\eta_v^{\vee}$ and $(\eta^{\ell}_{\theta,v})^{\vee}$, 
$$\left\|f^{\ell}_{\theta,v} - (\eta^{\ell}_{v, \theta}\eta_v (\psi_{\theta} f)^{\wedge} )^{\vee}\right\|_{L^{\infty}} \leq {\rm RapDec}(R) \|f\|_{L^2}. $$

Therefore, one has the decomposition
$$f= \sum_{\theta,\,v,\,\ell} f_{\theta,v}^\ell +{\rm Err}, \text{~where~}\|{\rm Err}\|_{L^{\infty}} \leq {\rm RapDec}(R)\|f\|_{L^2}.$$

The functions $\left\{f_{\theta,v}^\ell\right\}$  are almost orthogonal. For any set $\mathcal{T}$ of $(\theta,\,v,\,\ell)$, one has 
\[
\left\|\sum_{(\theta,v,\ell)\in \mathcal{T}}f_{\theta,v}^\ell\right\|_{L^2}^2\sim\sum_{(\theta,v,\ell)\in \mathcal{T}}\left\|f_{\theta,v}^\ell\right\|_{L^2}^2.
\]

When restricted inside a large ball $B_R$ centered at the origin with radius $R$, $Ef_{\theta,v}^\ell$ is essentially supported on a thin tube $T_{\theta,v}^\ell$ of length $R^{\delta}$ in the \emph{mini direction} $M(\theta) = (\xi_{\theta}, 1)$, $R$ in the \emph{long direction} $L(\theta)=(\xi_{\theta},-1)$, and $R^{1/2+\delta}$ in the rest of the directions. More precisely, $T^{\ell}_{\theta,v}$ can be identified with the $R^{\delta}$-neighborhood of the Minkowski sum  $P^{\ell}_{\theta,v}+ RL(\theta)$, where $P^{\ell}_{\theta,v}$ is viewed as a subset of $\mathbb{R}^n$ with the $n$th coordinate zero and by an abuse of notation, $RL(\theta)$ means the line segment $\{t L(\theta): 0\leq t\leq R\}$. Note that there is a one-to-one correspondence between the tubes $T^\ell_{\theta,v}$ and the parallelotopes $N_{R^\delta}(P^{\ell}_{\theta,v}+ RL(\theta))$ and they are of comparable sizes: $\frac{1}{10}T^\ell_{\theta,v}\subset N_{R^\delta}(P^{\ell}_{\theta,v}+ RL(\theta)) \subset 10 T^\ell_{\theta,v}$. Therefore, we do not distinguish them in the following.

The following lemma shows that $Ef^\ell_{\theta,v}$ is essentially supported on $T^\ell_{\theta,v}$.

\begin{lemma}
	If $x\in B_R\setminus T^{\ell}_{\theta,v}$, then 
	$$|Ef^{\ell}_{\theta,v}(x)|\leq {\rm RapDec}(R)\|f\|_{L^2}.$$
\end{lemma}
\begin{proof}
	Let $h$ be a function on $\mathbb{R}^{n-1}$ satisfying $\supp h\subset 2\bar{B}^{n-1}\setminus B^{n-1}$. Then the Fourier transform $\mathcal{F}$ of $Eh$ in $\mathbb{R}^n$ can be written as 
	$$\mathcal{F}(Eh)(\xi, \xi_n) = h(\xi) \delta_{\xi_n=|\xi|}.$$
	
	Similarly, $\mathcal{F}\big(E(hg)\big)= h(\xi) \delta_{\xi_n=|\xi|} \cdot g(\xi)$, hence 
	$$E(hg)= Eh \ast [ \hat{g}\delta_{x_n=0}]$$ where $\hat{g}$ denotes the Fourier transform of $g$ in $\mathbb{R}^{n-1}$. 
	
	By choosing $h=\widetilde{\psi}_{\theta}$ and $g= (\eta^{\ell}_{\theta,v} \eta_v (\psi_{\theta} f)^{\wedge})^{\vee}$, one can write $Ef^{\ell}_{\theta,v}=E(hg)$. Note that $\hat{g}$ is supported on $P^{\ell}_{\theta,v}$. 
	
	We use stationary phase to estimate $Eh$. Inside $B_R$, $|Eh(x)| \leq {\rm RapDec}(R)$ if $x$ is outside a thin tube $T^{0}_{\theta}$ centered at the origin, of length $R$ in the long direction $L(\theta)$, $R^{\delta}$ in the mini direction $M(\theta)$, and $R^{1/2+\delta}$ in the rest of directions. Recall that $T^{\ell}_{\theta,v}$ can be identified with the $R^{\delta}$-neighborhood of the Minkowski sum $P^{\ell}_{\theta,v} +RL(\theta)$, which contains the Minkowski sum $P^{\ell}_{\theta,v}+T^{0}_{\theta}$, the support of $E(hg)$. Hence the desired result follows. Again, here $P^{\ell}_{\theta,v}$ is considered as a subset of $\mathbb{R}^n$ with the $n$th coordinate being zero. 
\end{proof}

Let $\mathcal{T}$ be a collection of wave packets of $f$. In our argument, oftentimes the terminology $f$ is \emph{concentrated on wave packets from $\mathcal{T}$} is used, which means that
\[
\sum_{(\theta,v,\ell)\notin\mathcal{T}} \|f_{\theta,v}^\ell\|^2_{L^2}\lesssim {\rm RapDec}(R)\|f\|_{L^2}.
\]

\section{A $k$-broad estimate for the cone}\label{SecBroad}
Using a standard $\epsilon$-removal trick \cite{Tao1}, one can reduce the desired global Fourier extension estimate
\[
\|Ef\|_{L^p(\mathbb{R}^n)}\lesssim_p \|f\|_{L^p(2\bar{B}^{n-1}\setminus B^{n-1})},\quad \forall p \text{ satisfies (\ref{p range})},
\]to the following local version:
\begin{equation}\label{eq: 1}
\|Ef\|_{L^p(B_R^n)}\lesssim_{p,\epsilon} R^\epsilon \|f\|_{L^p(2\bar{B}^{n-1}\setminus B^{n-1})},\quad \forall \epsilon>0,\,\forall R >1,\quad \forall p \text{ satisfies (\ref{p range})}.
\end{equation}(Indeed, the range (\ref{p range}) is open, and the $\epsilon$-removal trick enables one to conclude the global estimate at all $p>\bar{p}$ from the local estimate at $p=\bar{p}$.) 

In order to study estimates of the form (\ref{eq: 1}) with $E$ replaced by the Fourier extension operator from the paraboloid, Guth \cite{Guth2} introduced a useful strategy that decomposes $Ef$ restricted on $B_R$ into a \emph{broad} part and a \emph{narrow} part. The narrow part is locally supported in some lower dimensional subspace of $\mathbb{R}^n$ and can be treated using decoupling \cite{BD}  and induction on spacial scales. He thus reduced (\ref{eq: 1}) to the estimate of the broad part and successfully derived (\ref{eq: 1}) for a large range of $p$.

In this article, we prove (\ref{eq: 1}) for the cone extension operator following the same strategy. The narrow part of $Ef$ will be treated in Section \ref{linearmain}. In this section, together with Sections \ref{strategy} and \ref{sec: proof}, we deal with the broad part by formulating a $k$--broad norm and by proving a general $k$-broad estimate (Theorem \ref{kBroadThm}).
For many arguments in the following sections, the $k$-broad norm $BL^p_k$ behaves almost the same as the $L^p$ norm.

Fix a large constant $R$ and $K$ such that $K\ll R$. We decompose $2\bar{B}^{n-1}\setminus B^{n-1}$ in the frequency space into sectors $\tau$ of dimension $1\times K^{-1}\times\cdots\times K^{-1}$, i.e. length $1$ in the radial direction and $K^{-1}$ in the rest of the directions. Using a smooth partition of unity subordinate to the cover $\{\tau\}$, one writes $f=\sum_{\tau} f_\tau$ where $f_\tau=f\chi_\tau$. 

Let $G(\tau)=\bigcup_{\theta\subset\tau} L(\theta)$. Here, recall that $\theta$ is a sector of $2\bar{B}^{n-1}\setminus B^{n-1}$ of  angular radius $R^{-1/2}$, and $L(\theta)$ denotes the long direction of the wave packets determined by $\theta$. Then $G(\tau)\subset S^{n-1}$ is contained in a spherical cap with radius $\approx K^{-1}$, representing possible long directions of wave packets in $Ef_{\tau}$. For any subspace $V\subset \mathbb{R}^n$, we adopt the notation $\text{Angle}(G(\tau),V)$ to denote the smallest angle between any non-zero vectors $v\in V$ and $v'\in G(\tau)$. 

In the physical space, we decompose the ball $B_R\subset\mathbb{R}^n$  into small balls $B_{K^2}$. For each $B_{K^2}\subset B_R$, consider $\int_{B_{K^2}}|Ef_{\tau}|^p$ for every $\tau$. 

Heuristically, we say $Ef$ is \emph{$k$-narrow at $B_{K^2}$} if there exists $\Gamma$, the $K^{-2}$-neighborhood of some $(k-2)$-dimensional linear subspace of $\mathbb{R}^{n-1}$, such that $\int_{B_{K^2}} |Ef|^p$ is dominated by $\int_{B_{K^2}} |Ef_{\Gamma}|^p$, where $f_\Gamma$ is the restriction of $f$ on $\Gamma$.  If $Ef$ is not $k$-narrow at $B_{K^2}$, then we say it is \emph{$k$-broad at $B_{K^2}$} and we would have 
\begin{equation}\label{kbroad heuristic}
\int_{B_{K^2}} |Ef|^p \leq K^{O(1)} \int_{B_{K^2}} \sup_{\substack{\tau_1,\cdots,\tau_k:\\G(\tau_1)\wedge\cdots \wedge G(\tau_k)\gtrsim K^{-O(1)}}}\left(\Pi_{j=1}^{k} |Ef_{\tau_j}|^{\frac{p}{k}}\right).
\end{equation}In the above, $G(\tau_1)\wedge\cdots \wedge G(\tau_k)$ denotes the infimum of the wedge product $L(\theta_1)\wedge\cdots \wedge L(\theta_k)$ over all choices of sectors $\theta_j\subset \tau_j$ of angular radius $R^{-1/2}$, $j=1,\cdots,k$.

The $k$-broad norm of $Ef$, roughly speaking, will be defined as the sum of the right hand side of  inequality~(\ref{kbroad heuristic}) over those $B_{K^2}$ where $Ef$ is $k$-broad. However, in order to make the argument rigorous, we need a more technical definition of $k$-broad norm that carries similar heuristics.

Here are the details. For a fixed parameter $1<A\lesssim K^{\epsilon}$, define
\begin{equation}\label{kBroadDef}
\mu_{Ef}(B_{K^2}):=\min_{V_1,\ldots,V_A\,(k-1)\text{-subspace of}\,\,\mathbb{R}^n}\left(\max_{\tau:\,\text{Angle}(G(\tau),V_a)>K^{-2},\forall a}\int_{B_{K^2}}|Ef_\tau|^p\right).
\end{equation}Then for any open set $U$ being the union of some balls $B_{K^2}$, we define the \emph{$k$-broad} part of $\|Ef\|_{L^p(U)}$ by
\[
\|Ef\|_{BL^p_{k,A}(U)}^p:=\sum_{B_{K^2}\subset U}\mu_{Ef}(B_{K^2}).
\]In fact, if defined on each $B_{K^2}$ as a constant multiple of the Lebesgue measure, $\mu_{Ef}$ can be extended to be a measure on $B_R$. In particular, $\mu_{Ef}(B_R)=\|Ef\|_{BL^p_{k,A}(B_R)}^p$. Note that a similar measure is used in \cite{Guth2} for the study of the broad part of the extension operator from the paraboloid. There, a similar quantity of $\mu_{Ef}(B_{K^2})$ is defined but with a different angle condition $\text{Angle}(G(\tau),V_a)>K^{-1}$. Our angle condition is more relaxed hence makes the broad estimate slightly more difficult. However, this change is necessary for the cone; later in the narrow case (Section \ref{linearmain}), one needs to ensure that there are not too many sectors $\tau$ whose corresponding long directions are near a low dimensional subspace $V\subset \mathbb{R}^n$. We leave the more detailed discussion on why $K^{-2}$ would be enough to Section \ref{linearmain}.

The parameter $A$ is introduced to make the norm $BL^p_{k,A}$ behave more like a regular $L^p$ norm. In particular, it satisfies the following triangle inequality and H\"older's inequality, which are adapted directly from Lemma 4.1 and 4.2 of \cite{Guth2}. Note that even though we are working with the cone and with a different angle condition, the same arguments in \cite{Guth2} still work. We omit their proofs.

\begin{lemma}[Triangle inequality]\label{lem: tri}
Suppose that $1\leq p<\infty$, $f=g+h$ and $A=A_1+A_2$, where $A, A_i$ are nonnegative integers. Then
\[
\|Ef\|_{BL^p_{k,A}(U)}\lesssim \|Eg\|_{BL^p_{k,A_1}(U)}+\|Eh\|_{BL^p_{k,A_2}(U)}.
\]
\end{lemma}
\begin{lemma}[H\"older's inequality]\label{lem: Holder}
Suppose $1\leq p,p_1,p_2<\infty$, and $0\leq \alpha_1,\alpha_2\leq 1$ obey $\alpha_1+\alpha_2=1$ and
\[
\frac{1}{p}=\frac{\alpha_1}{p_1}+\frac{\alpha_2}{p_2}.
\]Suppose that $A=A_1+A_2$, then
\[
\|Ef\|_{BL^p_{k,A}(U)}\leq \|Ef\|^{\alpha_1}_{BL^{p_1}_{k,A_1}(U)}\|Ef\|^{\alpha_2}_{BL^{p_2}_{k,A_2}(U)}.
\]
\end{lemma}In order to be able to apply these inequalities many times in the argument (say, $O_\epsilon(1)$ times), one needs to choose $A$ sufficiently large (depending on $\epsilon$). The relation between the parameters $K, A, R$ is the following:
$$1\ll A \lesssim K^{\epsilon} \lesssim R^{\epsilon^2}.$$

The main result of this section is the following:
\begin{theorem}\label{kBroadThm}
For any $2\leq k\leq n$ and any $\epsilon>0$, there is a large constant $A$ so that
\begin{equation}\label{kBroad}
\|Ef\|_{BL^p_{k,A}(B_R)}\lesssim_{\epsilon} R^\epsilon \|f\|_{L^2(2\bar{B}^{n-1}\setminus B^{n-1})}
\end{equation} holds for any $p\geq \bar{p}(k,n):=2\cdot\frac{n+k}{n+k-2}$.
\end{theorem}

Theorem \ref{kBroadThm} is a weak version of the $k$-linear cone restriction conjecture, which says that if $U_1,\ldots,U_k\subset 2\bar{B}^{n-1}\setminus B^{n-1}$ are transversal, i.e. $|L(\theta_1)\wedge\ldots\wedge L(\theta_k)|\gtrsim 1$ for any choices of $\theta_j\subset U_j$, and $f_j$ is supported in $U_j$, $1\leq j\leq k$, then
\begin{equation}\label{klinear}
\left\|\prod_{j=1}^k |Ef_j|^{1/k}\right\|_{L^p(B_R)}\lesssim R^\epsilon \prod_{j=1}^k\|f_j\|_{L^2(2\bar{B}^{n-1}\setminus B^{n-1})}^{1/k}.
\end{equation}
This has been proven in \cite{Wolff} and \cite{BCT} in the case $k=2$ and $k=n$ respectively. When $3\leq k\leq n-1$, it is unknown whether the $k$-linear cone restriction holds true. The only progress towards it that the authors are aware of is due to Bejenaru \cite{Bejenaru1, Bejenaru2}, where some sharp (up to the endpoint) $k$-linear restriction estimate was obtained for a class of hypersurfaces with curvature including $(k-1)$-conical surfaces using very different methods. Even though being a weaker result, (the corresponding version of) the $k$-broad estimate has been shown by Guth in \cite{Guth1, Guth2} to be sufficient for obtaining linear restriction estimates for the paraboloid. This follows from an adapted argument of Bourgain and Guth \cite{BG}, where a method converting multilinear restriction estimates into linear restriction estimates is introduced. In this sense, the core power of the $k$-linear restriction can be captured by the $k$-broad estimate, which inspired us to take a similar path in our proof and suggests possible further applications in other problems. 

In the next two sections, we prove Theorem \ref{kBroadThm}. Similarly as for the paraboloid, we apply the method of polynomial partitioning, which exploits the algebraic structure of the broad part of $|Ef|$. We will emphasize the  differences between the cases of the paraboloid and the cone, while only sketch the part of the proof where the argument for the paraboloid in \cite{Guth2} applies equally well in our problem.  

In Section \ref{strategy}, we recall some background of polynomial partitioning, provide an outline of the argument, and identify the main difficulties. Then, in Section \ref{sec: proof}, instead of directly proving Theorem \ref{kBroadThm}, we in fact prove a stronger inductive estimate (Theorem \ref{theorem8.1} below) that involves all intermediate dimensions $1\leq m\leq n$, which in particular recovers Theorem \ref{kBroadThm} at $m=n$. This strengthening is necessary in order for us to tackle the issues that arise over the course of induction and was also the strategy taken in \cite{Guth2}. 

\section{Outline of polynomial partitioning}\label{strategy}
Polynomial partitioning has been a powerful tool widely used in the study of restriction problems. It originated from the work of Guth--Katz \cite{GK} in their resolution of the Erd\"{o}s distinct distance conjecture in discrete geometry, and was introduced to the continuous setting, particularly for the restriction estimates for the paraboloid, by Guth \cite{Guth1, Guth2}. Briefly speaking, it is a strategy of divide and conquer; it begins with identifying a polynomial whose zero set partitions the mass of $\|Ef\|_{BL^p_{k,A}}$ into pieces. It thus suffices to estimate the part of $\|Ef\|_{BL^p_{k,A}}$ restricted in each small piece, and the part of $\|Ef\|_{BL^p_{k,A}}$ that is restricted near the zero set of the polynomial. Both situations turn out to be suitable for performing an induction type argument.

\subsection{Tools from algebraic geometry}

Given a polynomial $P$ on $\mathbb{R}^n$, its zero set is denoted by $Z(P)$. The basic partitioning theorem our argument will rely on is the following.
\begin{theorem}\cite[Theorem 1.4]{Guth1}\label{thm: pp}
Suppose that $W\geq 0$ is a nonzero $L^1$ function on $\mathbb{R}^n$. Then for each $D$ there exists a non-zero polynomial $P$ of degree at most $D$ such that $\mathbb{R}^n\setminus Z(P)$ is a union of $\sim D^n$ disjoint open sets $O_i$, and
\begin{equation}\label{thmpp cells}
\int_{O_i} W = \int_{O_j} W,\quad \forall i,j.
\end{equation}
\end{theorem}

We would want the zero sets of the partitioning polynomials  that appear in our proof to be smooth and regular, so that locally they can be well approximated by their tangent planes. To ensure this, we choose to work with varieties that are \emph{transverse complete intersections}. The following definition is borrowed from \cite[Section 5]{Guth2}.

\begin{definition}
Fix integer $m\in [1,n]$ and let $P_1,\ldots, P_{n-m}$ be polynomials on $\mathbb{R}^n$ whose common zero set is denoted by $Z(P_1,\ldots, P_{n-m})$. The variety $Z(P_1,\ldots, P_{n-m})$ is called a transverse complete intersection if
\[
 \nabla P_1(x)\wedge  \cdots \wedge\nabla P_{n-m}(x)\neq 0,\quad \forall x\in Z(P_1,\ldots, P_{n-m}).
\]
Define the degree of the transverse complete intersection as $\max_{j=1, \dots, n-m} \deg P_j$. 
\end{definition}A transverse complete intersection $Z(P_1,\ldots, P_{n-m})$ is a smooth $m$-dimensional manifold. 

\begin{remark}\label{rmk: pp}
Theorem \ref{thm: pp} does not guarantee that $Z(P)$ is a transverse complete intersection. After a small pertubation and using Sard's theorem, we could make $Z(P)$ a transverse complete intersection while changing ``$=$'' in equation~\ref{thmpp cells} to ``$\sim$''.  We refer the reader to Lemma 5.1 and Theorem 5.5 of \cite{Guth2} for details. 
\end{remark}
The information of $Ef$ is mostly carried by its wave packets. It is thus useful to understand how a wave packet may intersect a variety. 

In our argument, sometimes one needs to control the number of times a wave packet can cross a variety $Z$ transversally, hence the following result becomes helpful.
\begin{lemma}\cite[Lemma 5.7]{Guth2}\label{lem: trans} 
Let $T$ be a cylinder of radius $r$ with central line $\ell$ and suppose that $Z=Z(P_1,\ldots,P_{n-m})\subset \mathbb{R}^n$ is a transverse complete intersection, where the polynomials $P_j$ have degree at most $D$. For any $\alpha>0$, define
\[
Z_{>\alpha}:=\{z\in Z:\, {\rm Angle}(T_z Z, \ell)>\alpha\}.
\]Then $Z_{>\alpha}\cap T$ is contained in a union of $\lesssim D^n$ balls of radius $\lesssim r\alpha^{-1}$.
\end{lemma}When applying the lemma, a typical choice is $r=R^{(1+\delta)/2}$ and $\alpha=R^{-1/2+\delta}$. Note that in the case of the cone, the wave packets are thin tubes which are even smaller than the cylinders $T$ in the lemma above, hence the same result holds true for the wave packets.

\subsection{Polynomial partitioning in $\mathbb{R}^n$}
We now apply the polynomial partitioning theorem to $\mu_{Ef}$, the measure that was defined via the broad norm of $Ef$ after (\ref{kBroadDef}). Let $B_R\subset \mathbb{R}^n$ be the fixed large ball as before. By Theorem \ref{thm: pp} and Remark \ref{rmk: pp}, for a large constant $D\lesssim_{\epsilon, m} 1$ to be determined later, there exists a (non-zero) polynomial of degree at most $D$ such that its zero set $Z$ divides $B_R\setminus Z$ into a disjoint union of $O(D^n)$ parts $O_i$ with comparable measure $\mu_{Ef}(O_i)\sim \frac{1}{D^n} \mu_{Ef}(B_R)$. 

Recall the wave packet decomposition $Ef=\sum_{\theta,v,\ell}Ef_{\theta,v}^\ell$, where each wave packet in the physical space is essentially supported in $T_{\theta,v}^\ell$, a thin tube of length $R$, radius $R^{\frac{1+\delta}{2}}$ and thickness $R^{\delta}$. In the simplified model where each $T_{\theta,v}^\ell$ is reduced to a line segment, $T_{\theta,v}^\ell$ intersects at most $D$ different parts $O_i$, which is much fewer than the total number of $O_i$'s. In other words, the wave packets passing through a fixed $O_{i_0}$ do not interact much with other $O_i$'s, which works in our favor when we do induction. However, unlike a line segment, a tube $T_{\theta, v}^\ell$ might intersect many more $O_i$'s. In order to apply the above heuristic, we need to first thicken $Z$ to a \emph{wall} $W$, which is defined as the $R^{\frac{1+\delta}{2}}$-neighborhood of $Z$. Let $\widetilde{O}_i:=O_i\setminus W$ be a \emph{cell}, then one has the partition
\[
B_R \subset W\sqcup \left(\bigsqcup_{i}\widetilde{O}_i \right),
\]and the fact that each $T_{\theta,v}^\ell$ intersects at most $D$ cells. 

Therefore, one has 
\[
\mu_{Ef}(B_R)= \sum_{i}\mu_{Ef}(\widetilde{O}_i)+\mu_{Ef}(W).
\]We say that we are in the \emph{Cellular} case if the first term dominates the right hand side of the above equality, the \emph{Algebraic} case if the second term dominates.

\subsubsection{Cellular case}\label{cell1}
This case can be treated in the same way as for the paraboloid, based on the fact that each tube intersects at most $O(D)$ cells. In fact, it holds even more easily since tubes in the cone case are thinner. Let $Ef_i = \sum_{T_{\theta,v}^\ell\cap \widetilde{O}_i\neq \emptyset} Ef_{\theta,v}^\ell$, one then has
$$\sum_{i} \|Ef_i\|_{L^2(B_R)}^2 \lesssim D \|Ef\|_{L^2(B_R)}^2.$$
By Plancherel, 
$$\sum_{i} \|f_i\|_{L^2}^2 \lesssim D \|f\|_{L^2}^2.$$
Combined with $\sum_{i} \mu_{Ef}(\widetilde{O}_i) \sim \mu_{Ef}(B_R)$, there exists at least one cell $\widetilde{O}_i$ (in fact true for most of the cells) such that both of the following estimates hold: 
\begin{align*}
\mu_{Ef}(B_R)&\lesssim D^n \mu_{Ef}(\widetilde{O}_i),\\
\|f_i\|_{L^2}^2&\lesssim \frac{1}{D^{n-1}} \|f\|_{L^2}.
\end{align*}
We cover $\widetilde{O}_i$ with finitely many balls of radius $R/2$ and induct on the radius of the ball $B_R$. The induction closes if $p>\frac{2n}{n-1}$. More precisely,
\begin{align*}
\mu_{Ef}(B_R)&\lesssim D^n \mu_{Ef}(\widetilde{O}_i) \lesssim D^n \sum_{B_{R/2}\subset B_R} \mu_{Ef}(\widetilde{O}_i\cap B_{R/2})\\
&\lesssim R^{\epsilon}D^n\|f_i\|_{L^2}^{p} \lesssim R^{\epsilon} D^{n-\frac{(n-1)p}{2}} \|f\|_{L^2}^p.
\end{align*}If $D$ is chosen sufficiently large, the power of $D$ dominates the implicit constant and the induction is closed.

\subsubsection{Algebraic case}

A tube can intersect the wall $W$ in two different ways, either cutting across $W$ or nearly tangent to $Z$. 
\begin{definition}\label{tangent}
Let $Z_0$ be an $m$-dimensional variety in $\mathbb{R}^n$. A tube $T_{\theta,v}^\ell$ is said to be $\gamma$-tangent to $Z_0$ in $B_R$ if 
\[
T_{\theta,v}^\ell\subset N_{\gamma R}(Z_0)\cap 2 B_R
\]and for all $z\in Z_0\cap N_{10\gamma R}(T_{\theta,v}^\ell)\cap 2B_R$ there holds
\[
\text{Angle}(T_z Z_0,L(\theta))\leq\gamma,\qquad \text{where $L(\theta)$ denotes the long direction of $T_{\theta,v}^\ell$.}
\]
\end{definition}
Fix $\delta>0$. (In this outline section, $\delta$ is the same as the one in Section 2 and  is much smaller than $\epsilon$. In later sections,  the $\delta$ in $\mathbb{T}_{\text{tang}}$ and $\mathbb{T}_{\text{trans}}$ will be $\delta_m$, depending on $\dim Z_0 = m$, as in Theorem \ref{theorem8.1}.) If a tube intersects $W$, then we say it crosses $W$ \emph{transversally} if it is not $R^{-1/2+\delta}$-tangent to $Z$. Denote 
\[
\mathbb{T}_{\text{trans}}:=\{(\theta,v,\ell):\, T_{\theta,v}^\ell\,\,\text{crosses}\,\,W\,\,\text{transversally in}\,\, B_R\},
\]
\[
\mathbb{T}_{\text{tang}}:=\{(\theta,v,\ell):\, T_{\theta,v}^\ell\,\,\text{is}\,\, R^{-1/2+\delta}\,\,\text{-tangent to}\,\, Z\,\,\text{in}\,\,B_R\},
\]and let
\[
f_{\text{trans}}:=\sum_{(\theta,v,\ell)\in\mathbb{T}_{\text{trans}}}f_{\theta,v}^\ell,\qquad f_{\text{tang}}:=\sum_{(\theta,v,\ell)\in\mathbb{T}_{\text{tang}}}f_{\theta,v}^\ell.
\]By triangle inequality of the broad norm (Lemma \ref{lem: tri}), there are two different cases to consider depending on which type of wave packets make the most contribution to $\mu_{Ef}(W)$:
\begin{itemize}
\item{\emph{Algebraic transversal}:} if $\mu_{Ef_{\text{trans}}}(W)\gtrsim \mu_{Ef}(B_R)$;
\item{\emph{Algebraic tangential}:} if $\mu_{Ef_{\text{tang}}}(W)\gtrsim \mu_{Ef}(B_R)$.
\end{itemize}

The transversal case can be dealt with by induction. Cover $W$ with balls $\{B_j\}$ of radius $\rho:=R^{1-\delta}$ and notice that $T_{\theta,v}^\ell \in\mathbb{T}_{\text{trans}}$ crosses $W$ transversally in at most $\lesssim D^n$ different $B_{j}$'s according to Lemma \ref{lem: trans} (by taking in the lemma $r=R^{(1+\delta)/2}$ and $\alpha=R^{-1/2+\delta}$). Fix a $B_j$ and let $Ef_j :=\sum_{T_{\theta,v}^\ell\in\mathbb{T}_{\text{trans}}, T_{\theta,v}^\ell\cap B_j\neq \emptyset}Ef_{\theta,v}^\ell$. By inducting on scales, one obtains
\begin{align*}
\mu_{Ef_{\text{trans}}}(W)&\leq \sum_{B_{j}} \mu_{Ef_{\text{trans}}}(B_{j}\cap W) \\
&\leq \sum_{B_j}\mu_{Ef_{j}}(B_j\cap W) + \text{RapDec}(R)\|f\|_{L^2}^p\\
&\lesssim \sum_{B_j} {\rho}^{\epsilon} \|f_j\|_{L^2}^p + \text{RapDec}(R)\|f\|_{L^2}^p\\ 
&\lesssim  \rho^{\epsilon}D^{pn/2}\|f\|_{L^2}^p\lesssim R^{\epsilon} \|f\|_{L^2}^p.
\end{align*}
Since $D\lesssim_{\epsilon, m} 1$, we can choose $R$ sufficiently large  so that $R^{\delta \epsilon} \gg D^{pn/2}$, hence the last inequality holds. Note that this argument is still the same as in the paraboloid problem. 

Things begin to change in the tangential case, where the cone restriction problem becomes different from the paraboloid one. Because of the lack of curvature on the straight lines on the cone, we choose to work with wave packets that are thinner than the ones for the paraboloid, which however results in more wave packets lying inside the $R^{\frac{1+\delta}{2}}$-neighborhood of a variety tangentially. 

The main strategy in this case is to perform another polynomial partitioning inside $W$, look into the cellular, transversal  and tangential cases at the next level, and repeat. At each step, the dimension of the variety (denoted as $Z$ again) that the wave packets are tangent to is reduced by $1$. And the iteration stops when $\dim Z < k$ according to the following lemma.
\begin{lemma}\label{base}
If $Ef$ is $R^{-1/2+\delta}$-tangent to a variety $Z$ of degree $O(1)$ and dimension $k-1$, then
$$\|Ef\|_{BL^p_{k,A}(B_R)} \leq \text{RapDec}(R)\|f\|_{L^2}.$$
\end{lemma}
\begin{proof}
Fix any ball $B$ of radius $R^{\frac{1+\delta}{2}}$ inside the $R^{\frac{1+\delta}{2}}$-neighborhood of $Z$, for any $x\in B\cap Z$ and any $T_{\theta,v}^\ell \cap B \neq \emptyset$, by the assumption the long direction of $T_{\theta,v}^\ell$ lies 
inside the $R^{-1/2+\delta}$-neighborhood of the tangent space $T_x Z$. Since the dimension of $T_x Z$ is $k-1$, by the definition of the $k$--broad norm, there follows
$$\|Ef\|_{BL^p_{k,A}(B)} \lesssim \text{RapDec}(R) \|f\|_{L^2}. $$
\end{proof}

Guth \cite{Guth2} applied this strategy for the paraboloid. The key idea is that if $Ef$ is tangential to a $m$-dimensional variety $Z$, then  one can essentially treat $Z$ as $\mathbb{R}^m$ and make use of a so-called \emph{equidistribution} property. Morally speaking, the \emph{equidistribution} property says that $|Ef|$ is roughly a constant function locally along the normal direction of $Z$. This, however, is not true for the cone. The main ingredient in our proof is to establish this equidistribution property after removing some negligible part of $Ef$. 

\section{Main inductive argument}\label{sec: proof}
In this section, we prove the broad estimate (Theorem \ref{kBroadThm}), which will be a consequence of a more general result (Theorem \ref{theorem8.1} below). As mentioned at the end of the previous section, we will apply polynomial partitioning iteratively on a sequence of sub-varieties in $\mathbb{R}^n$ of various dimensions. 

To begin with, we discuss how polynomial partitioning, introduced in the previous section on $\mathbb{R}^n$, can be extended to partition a general sub-variety in $\mathbb{R}^n$. 

\begin{theorem}\cite{Guth2}\label{thm: ppsub}
	Fix $r\gg1$, $d\in \mathbb{N}$ and suppose $F \in L^1(\mathbb{R}^n)$ is non-negative and supported on $B_r\cap N_{r^{1/2+\delta}}\mathbf{Z}$ for some $0<\delta\ll 1$, where $\mathbf{Z}$ is an $m$--dimensional transverse complete intersection of degree at most $d\lesssim_{\epsilon,m} 1$. Then, there exists $D=D(\epsilon, d)$ with $d \lesssim D^{\delta^2} \lesssim_{\epsilon, m} 1$ such that at least one of the following cases holds:
	\begin{enumerate}
		\item \emph{(Cellular case)} There exists a polynomial $P$: $\mathbb{R}^n\rightarrow \mathbb{R}$ of degree $D$ satisfying the following properties. There exist $\sim D^m$ cells $O\subset \mathbf{Z}\setminus N_{r^{1/2+\delta}}Z(P)$ such that $O\subset B_{r/2}$ and 
		$$\int_O F \sim D^{-m} \int_{\mathbb{R}^n} F \qquad \text{~~ for~all~} O.$$
		Furthermore, each tube of length $r$, radius $r^{1/2+\delta}$ intersects at most $O(D)$ cells. 
		\item \emph{(Algebraic case)} There exists an $(m-1)$--dimensional transverse complete intersection $\mathbf{Y}$ of degree at most $O(D)$ such that 
		$$\int_{B_r\cap N_{r^{1/2+\delta}}\mathbf{Z}} F \lesssim \int_{B_r \cap N_{r^{1/2+\delta}}\mathbf{Y}} F. $$
		\end{enumerate}
\end{theorem}

Theorem~\ref{thm: ppsub} is  proved in Section 8.1 of \cite{Guth2} while not explicitly stated. We borrow the exact statement of Theorem~\ref{thm: ppsub} from \cite[Theorem 6.3]{HR}. We briefly sketch its proof here. 

One first decomposes $\mathbf{Z}$ into $O(1)$ pieces $\mathbf{Z}_j$  such that the tangent spaces at points in each piece $\mathbf{Z}_j$ form an angle of at most $1/100$ with a certain $m$-dimensional subspace $V_j$.  There exists a $\mathbf{Z}_j$ such that 
$$\int_{B_r\cap N_{r^{1/2+\delta}}\mathbf{Z}} F \lesssim \int_{B_r\cap N_{r^{1/2+\delta}}\mathbf{Z}_j} F.$$

Next, one looks at only $\mathbf{Z}_j$ and define the orthogonal projection $\pi:\mathbb{R}^n \rightarrow V_j$. Applying Theorem~\ref{thm: pp} with the function $W(y)=\int_{\pi^{-1}(y)} F$, one can partition $V_j =\mathbb{R}^m$ using a polynomial  $P_{V_j}$ of degree $D=D(\epsilon, d)$. 
 Let $P$ be the polynomial on $\mathbb{R}^n$ defined as $P(x)= P_{V_j}(\pi(x))$. We then apply the polynomial partitioning argument presented in the last section.  If it is the cellular case for $W$ defined on $V_j$ with $P_{V_j}$, then we would obtain the cellular case in Theorem~\ref{thm: ppsub} with polynomial $P$.  Otherwise it is the algebraic case for $W$, so $\sim 1$ fraction of $F$ is concentrated on the $r^{1/2+\delta}$--neighborhood of $\mathbf{Z}\cap Z(P)$. We then apply Remark~\ref{rmk: pp} to fine tune $\mathbf{Z}\cap Z(P)$ into a transverse complete intersection $\mathbf{Y}$ after a small perturbation.

 \vspace{5pt}

Instead of proving Theorem \ref{kBroadThm} directly, we prove the following stronger estimate, which is similar to \cite[Theorem 8.1]{Guth2} and is more suitable for induction. 
\begin{definition}\label{def: tang}
 Let $S$ be a transverse complete intersection of degree $D_1\sim O(1)$ and of dimension $m<n$ inside $B_R$ ($S$ is understood as $S\cap B_R$ if it is not completely contained in $B_R$). Define
\begin{equation}\label{tangentset}
\mathbb{T}_S:=\left\{(\theta,v,\ell):\, T_{\theta,v}^\ell\,\,\text{is}\,\,R^{-1/2+\delta_m}\,\text{-tangent to S in}\,\,B_R\right\},
\end{equation}
where $\delta_m\geq0$ is a fixed small parameter for each dimension $m$, which is chosen later. 
\end{definition}
\begin{theorem}\label{theorem8.1}
For $\epsilon>0$, there exist small parameters $0<\delta\ll \delta_{n-1}\ll \cdots\ll \delta_1\ll \delta_0\ll \epsilon$ and large parameter $\bar{A}$ such that the following holds. Let $1\leq m\leq n$ and $\mathbf{Z}=Z(P_1,\ldots,P_{n-m})$ be a transverse complete intersection with $\text{Deg}\,P_i\leq D_{\mathbf{Z}} \lesssim_{\epsilon,m} 1 $. Suppose that $f$ is concentrated on wave packets from $\mathbb{T}_{\mathbf{Z}}$ as in Definition~\ref{def: tang}. Then for any $2\leq k\leq n$,  $1\leq A\leq\bar{A}$ and radius $R\geq 1$,
\[
\|Ef\|_{BL^p_{k,A}(B_R)}\lesssim_{K,\epsilon,m,D_{\mathbf{Z}}} R^\epsilon R^{\delta(\log\bar{A}-\log A)}R^{-e+\frac{1}{2}}\|f\|_{L^2}
\]whenever $2\leq p\leq p(m,k):=2\cdot\frac{m+k}{m+k-2}$ where $e:=\frac{1}{2}(\frac{1}{2}-\frac{1}{p})(n+k)$.
\end{theorem}
Observe that when $m=n$ and $\mathbf{Z}=\mathbb{R}^n$, by taking $A=\bar{A}$ and $p=p(n,k)$ one computes $-e+1/2=0$, which implies Theorem \ref{kBroadThm}. We also remark that for $p=2$, Theorem \ref{theorem8.1} follows quickly from a similar $L^2$ estimate as in Lemma 3.2 of \cite{Guth2}:
\begin{equation}\label{L2}
\|Ef\|_{BL^2_{k,A}(B_R)}^2\lesssim R\|f\|_{L^2}^2.
\end{equation}

By interpolation and H\"older's inequality of the broad norm (Lemma \ref{lem: Holder}), Theorem \ref{theorem8.1} will thus follow from the endpoint case $p=p(m,k)$, which we prove by induction next. Due to technical reasons, we choose to present the proof by induction rather than iteration similarly as in \cite{Guth2}. This is also why we need to prove a stronger result that concerns algebraic varieties of all intermediate dimensions $m$, which is more suitable for induction.

The rest of the section is devoted to the proof of Theorem \ref{theorem8.1}.

We will repeatedly use the strategy introduced in Section \ref{strategy}. More precisely, the plan is to induct on the dimension $m$, the radius $R$, and the parameter $A$. 

It is easy to see that the base case $m=k-1$ (for all $R$ and $A$) follows from Lemma \ref{base}. If $A=1$, then by choosing $\bar{A}$ large enough, the desired estimate follows from the trivial $L^1\to L^\infty$ estimate of the extension operator $E$. If $R$ is small, then the desired estimate can be deduced by choosing the implicit constant sufficiently large. Now suppose the desired estimate holds true if we decrease the dimension $m$, the radius $R$, or $A$. 

Recall that $\mathbf{Z}$ is a transverse complete intersection of dimension $m$. We first apply  Theorem~\ref{thm: ppsub}, then it suffices to discuss the cellular case and the algebraic case one by one.
\subsection{The cellular case}\label{secsec: cell}



Let $O$ be a cell, and define  $f_O = \sum_{(\theta,v,\ell)\in \mathbb{T}_O}f^\ell_{\theta,v}$, where
$$
\mathbb{T}_O:=\big\{(\theta,v,\ell):T_{\theta,v}^\ell \cap O\neq \emptyset\big\}.
$$
Since we are in the cellular case,  for $\sim D^m$ cells $O$, 
$$
\|Ef\|^{p}_{BL^{p}_{k,A} (B_R)} \lesssim D^m \|Ef\|^{p}_{BL^{p}_{k,A} (O)} \lesssim D^m \|Ef_O\|^{p}_{BL^{p}_{k,A} (B_R)}.
$$
Moreover, by orthogonality and the geometric observation that each $(\theta,v,\ell)$ belongs to $\lesssim D$ collections $\mathbb{T}_O$, 
$$
\sum_i \|f_i\|_{L^2}^2 \lesssim D \|f\|_{L^2}^2.
$$
Therefore, by the same argument as in Subsection \ref{cell1}, the induction for the non-algebraic case closes since $p=p(m,k)>\frac{2m}{m-1}$. 

\begin{remark}
In fact, when proving the case $m=k$, one needs to first prove the slightly larger endpoint case $p=p(m,m)+\delta$ and then interpolate. This is to make sure that the induction on scales argument treating the cellular case described above can close. More precisely, this slight change will produce a gain of $D^{-\delta}$ at the end of the cellular case inductive argument, for some $D=D(\epsilon, D_{\mathbf{Z}})$. By choosing $D$ sufficiently large, one then closes the induction. We omit the separate discussion of this special case as the issue can be handled in the exact same way as in \cite[Section 8.1, bottom of page 38]{Guth2}.
\end{remark}

\subsection{The algebraic case}

Fix $p\in [2, p(m,k)]$. Recall that in the algebraic case, there exists a transverse complete intersection $\mathbf{Y}$ of dimension $m-1$, defined using polynomials of degree $\leq D(\epsilon,D_{\mathbf{Z}})\lesssim_{\epsilon,m} 1$ such that 
\[
\mu_{Ef}(N_{R^{1/2+\delta_m}}(\mathbf{Y})\cap B_R)\gtrsim \mu_{Ef}(B_R).
\]

One first covers $B_R$  by smaller balls $B_j$ of radius $\rho$, where $\rho^{1/2+\delta_{m-1}}=R^{1/2+\delta_m}$. One has
\[
\|Ef\|^{p}_{BL^{p}_{k,A}(B_R)}\lesssim\sum_{j}\|Ef_j\|^{p}_{BL^{p}_{k,A}(B_j)}+{\rm RapDec}(R)\|f\|_{L^2}^{p},
\]where
\[
f_j:=\sum_{(\theta,v,\ell)\in\mathbb{T}_j} f_{\theta,v}^\ell,\quad \mathbb{T}_j:=\{(\theta,v,\ell):\,T_{\theta,v}^\ell\cap N_{R^{1/2+\delta_m}}(\mathbf{Y})\cap B_j\neq \emptyset\}.
\]

%

For each group $\mathbb{T}_j$, we further subdivide it into tubes that are tangent to $\mathbf{Y}$ and ones that are transverse to $\mathbf{Y}$. 

We say that $T_{\theta,v}^\ell \in \mathbb{T}_j$ is \emph{tangent} to $\mathbf{Y}$ in $B_j$ if
\begin{equation}\label{eqn:verify1}
T_{\theta,v}^\ell\cap B_j\subset N_{R^{1/2+\delta_m}}(\mathbf{Y}) \cap B_j =
N_{\rho^{1/2+\delta_{m-1}}}(\mathbf{Y}) \cap B_j
\end{equation}and for any non-singular point $y\in \mathbf{Y} \cap B_j\cap N_{10R^{1/2+\delta_m}}T_{\theta,v}^\ell$,
\begin{equation}\label{eqn:verify2}
 \text{Angle}(L(\theta),T_y \mathbf{Y})\leq \rho^{-1/2+\delta_{m-1}}.
\end{equation}

The groups of tangential and transversal wave packets are denoted by 
\[
\mathbb{T}_{j,{\rm tang}}:=\{(\theta,v,\ell)\in\mathbb{T}_j:\, T_{\theta,v,\ell} \text{ is tangent to } \mathbf{Y} \text{ in } B_j\},\quad \mathbb{T}_{j,{\rm trans}}:=\mathbb{T}_j\setminus \mathbb{T}_{j,{\rm tang}},
\]and let
\begin{equation}\label{eqn:trans}
f_{j,{\rm tang}}=\sum_{(\theta,v,\ell)\in\mathbb{T}_{j,{\rm tang}}} f_{\theta,v}^\ell,\quad f_{j, {\rm trans}}=\sum_{(\theta,v,\ell)\in\mathbb{T}_{j,{\rm trans}}} f_{\theta,v}^\ell.
\end{equation}Then by the triangle inequality (Lemma \ref{lem: tri}),
\[
\sum_{j}\|Ef_j\|^{p}_{BL^{p}_{k,A}(B_j)}\lesssim \sum_j\|Ef_{j,{\rm tang}}\|^{p}_{BL^{p}_{k,A/2}(B_j)}+\sum_j\|Ef_{j,{\rm trans}}\|^{p}_{BL^{p}_{k,A/2}(B_j)}.
\]In the following, we will estimate the contribution from the tangential wave packets and the transversal wave packets separately by induction on the dimension $m$, parameter $A$ and radius $R$. 

Before diving into the study of the two cases, we first discuss a common ingredient in their proofs: the relation between the wave packet decomposition of $Ef_j$ at the large scale $R$ and its wave packet decomposition inside ball $B_j$ at the small scale $\rho$. Understanding this relation is one of the main novelties of the article. Note that even though a similar discussion for the paraboloid can be found in \cite[Section 7]{Guth2}, many results there (for instance see Remark \ref{rmk: small wp} below) do not extend to the cone case, as the wave packet decomposition for the cone and the paraboloid are different.

\subsubsection{Adjusting the wave packet decomposition to a smaller ball}\label{subsec: largesmall}
Fix a small ball $B(y, \rho)\subset B(0,R)$ with $R^{1/2+\delta}<\rho <R$. Let $X=x-y$ and define 
$$\psi_y(\xi)= y_1\xi_1+\cdots y_{n-1}\xi_{n-1}+y_n |\xi|.$$
We also define the map $\widetilde{f}(\xi)= e^{i\psi_y(\xi)}f(\xi)$. 
Then $Ef(x)=E\widetilde{f}(X)$. 

Consider wave packet decomposition of $\widetilde{f}$ at scale $\rho$. In other words, for $E\widetilde{f}(X)$ defined on $B(0,\rho)$, write
$$\widetilde{f}= \sum_{\zeta, \, w, \,L} \widetilde{f}^{L}_{\zeta,w},$$
where each $\zeta$ is a small sector of $2\bar{B}^{n-1}\setminus B^{n-1}$ of radius $\rho^{-1/2}$, $w\in \rho^{1/2+\delta}\mathbb{Z}^{n-1}$, and $1\leq L \leq \rho^{1/2}$. The $(n-1)$-dimensional Fourier transform of each $\widetilde{f}^{L}_{\zeta, w}$ is essentially supported inside a thin plate $P^{L}_{\zeta,w}$ of side length $\rho^{1/2+\delta}$ and thickness $\rho^{\delta}$ in the ball of radius $\rho^{1/2+\delta}$ centered at $w$. The small wave packet $E\widetilde{f}^{L}_{\zeta,w}$ is essentially supported in a thin tube $T^{L}_{\zeta,w}$ of length $\rho$ in the long direction $L(\zeta)$, thickness $\rho^{\delta}$ in the mini direction $M(\zeta)$, and width $\rho^{1/2+\delta}$ in the rest of the directions. In the $X$ coordinate, the tube is contained in $B(0,\rho)$, while in the $x$ coordinates, the tube is translated to be in $B(y, \rho)$.

We would like to study how the original wave packet decomposition $f=\sum_{\ell, \theta,v} f^{\ell}_{\theta,v}$ is related to the new wave packet decomposition $\widetilde{f}=\sum_{L, \zeta,w} \widetilde{f}^{L}_{\zeta, w}$. 

For any $(\theta,v,\ell)$ such that $T_{\theta,v}^\ell\cap B(y,\rho)\neq \emptyset$, define $$\widetilde{\mathbb{T}}_{\theta,v, \ell} =\left\{(\zeta, w, L): \, {\rm Dist}(\theta, \zeta)\lesssim \rho^{-1/2}, \, {\rm Dist}(P^{L}_{\zeta, w}, \, P^{\ell}_{\theta,v}+P^{0}_{\theta} -\partial_{\xi} \psi_y (\xi_{\theta}))\lesssim R^{\delta}\right\}.$$
Recall that $\xi_{\theta}$ is the point on the central line of $\theta$ with $|\xi_{\theta}|=1$;  $P^0_{\theta}$ is the thin plate centered at the origin in $\mathbb{R}^{n-1}\times \{0\}$ of side length $R^{1/2+\delta}$, thickness $R^{\delta}$, with normal direction $\xi_{\theta}$; and $P^{\ell}_{\theta,v}$ is the essential support of $(f^{\ell}_{\theta,v})^{\wedge}$. 

\begin{lemma}\label{lm: small wp}
	$(f^{\ell}_{\theta,v})^{\sim} $ is concentrated on small wave packets from $\widetilde{\mathbb{T}}_{\theta,v, \ell}$. In other words,
\[
(f^{\ell}_{\theta,v})^{\sim}=\sum_{(\zeta,w,L)\in \widetilde{\mathbb{T}}_{\theta,v, \ell}} (g_{\zeta,w}^L)^\sim +{\rm RapDec}(R)\|f\|_{L^2},	
\]where $g=f^\ell_{\theta,v}$. 
\end{lemma}
\begin{proof}
	First, since $(f^{\ell}_{\theta,v})^{\sim} = e^{i\psi_{y}(\xi)} f^{\ell}_{\theta,v}$ is essentially supported on $\theta$, it is obviously concentrated in small wave packets $(\zeta, w,L)$ satisfying ${\rm Dist}(\theta, \zeta)\lesssim \rho^{-1/2}$. 
	
	Let $\phi_{\theta}$ be a bump function that is $1$ on $\theta$ and $0$ outside $2\theta$.
	Then the Fourier transform of $(f^{\ell}_{\theta,v})^{\sim} = e^{i\psi_{y}(\xi)} f^{\ell}_{\theta,v} \phi_{\theta}$ is 
	$$\big(e^{i\psi_{y}(\xi) }f^{\ell}_{\theta,v} \phi_{\theta}\big)^{\wedge} = (\phi_{\theta} e^{i\psi_{y}(\xi)})^{\wedge} \ast (f^{\ell}_{\theta,v})^{\wedge}.$$
	
	In the following, we apply stationary phase to show that $(\phi_{\theta} e^{i\psi_{y}(\xi)})^{\wedge}$ is rapidly decaying outside $-\partial_{\xi} \psi_{y}(\xi_{\theta})+P^{0}_{\theta}$. Then, it will follow that the Fourier transform of $(f^{\ell}_{\theta,v})^{\sim}$ is essentially supported in $P^{\ell}_{\theta,v} -\partial_{\xi} \psi_{y}(\xi_{\theta})+P^{0}_{\theta}$. Hence, the second distance condition in the definition of $\widetilde{\mathbb{T}}_{\theta,v, \ell}$ also holds true, which will complete the proof. 

To show that $(\phi_{\theta} e^{i\psi_{y}(\xi)})^{\wedge}$ rapidly decays outside $-\partial_{\xi} \psi_{y}(\xi_{\theta})+P^{0}_{\theta}$, we first Taylor expand  the phase function $$\psi_y(\xi) = \psi_y(\xi_{\theta})+ \partial_{\xi}\psi_{y}(\xi_{\theta})\cdot (\xi-\xi_{\theta}) + \psi_{y,{\rm tail}}(\xi).$$ Note that we can ignore the constant terms $\psi_y(\xi_{\theta})$ and $-\partial_{\xi}\psi_y(\xi_{\theta})\cdot \xi_{\theta}$. Let $\Phi_{\theta} = e^{i\psi_{y,{\rm tail}} }\phi_{\theta}$, we have 
$$ (\phi_{\theta} e^{i\psi_y(\xi)})^{\wedge}(x) =  e^{i \psi_y(\xi_{\theta}) - i \partial_{\xi}\psi_y(\xi_{\theta})\cdot \xi_{\theta}}\int e^{i\xi\cdot (x+\partial_\xi\psi_y(\xi_{\theta}))} \Phi_{\theta}(\xi) d\xi. $$ It thus remains to show that $\widehat{\Phi}_{\theta}$ is essentially supported on $P^0_{\theta}$. 

Up to a rotation, we might assume that $\xi_{\theta}=(0, \dots, 0, 1)$. Consider the change of variable $A: (\xi_1,\dots, \xi_{n-1})\mapsto (R^{-1/2}\xi_1, \dots, R^{-1/2} \xi_{n-2}, \xi_{n-1})$. Then $\phi_{\theta}(A\cdot)$ is supported on the unit ball and
\begin{align*}
\widehat{\Phi}_{\theta}(A^{-1}x)&=\int  e^{i  \xi \cdot x +i\psi_{y,{\rm tail}}(A\xi)} \phi_{\theta}(A\xi) \,dA\xi. 
\end{align*}By definition, we have $\partial_{\xi} (\psi_{y,{\rm tail}}(A\xi))= A[\partial_{\xi}\psi_{y}(A\xi) - \partial_{\xi}\psi_{y}(A\xi_{\theta})]$. In particular, this implies for all $\xi$ in the unit ball that
\[
|\partial_{\xi} (\psi_{y,{\rm tail}} (A\xi))| = y_n \left|\left(\frac{R^{-1}\xi_1}{|A\xi|} , \dots, \frac{R^{-1}\xi_{n-2}}{|A\xi|}, \frac{\xi_{n-1}}{|A\xi|}-1\right)\right|\lesssim 1.
\]Indeed, since $y_n\leq R$ and $|A\xi|\approx 1$, the first $n-2$ coordinates are bounded. The last coordinate is also bounded because $ \xi_{n-1}-|A\xi|=\xi_{n-1}-\sqrt{R^{-1}\xi_1^2+\cdots R^{-1}\xi_{n-2}^2 +\xi_{n-1}^2} \lesssim  R^{-1} |\xi|^2$. 

Therefore, when $x\gtrsim R^{\delta}$, $|\widehat{\Phi}_{\theta}(A^{-1}x)|\lesssim_N R^{-N}$ for all $N$. This proves the rapidly decaying property of $\widehat{\Phi}_{\theta}(x)$ outside $P^0_{\theta}$. 




\end{proof}

Let  $T_{\zeta,w}^L$ be a small thin tube $(\zeta, w, L) \in \widetilde{\mathbb{T}}_{\theta,v, \ell}$ in the $x$ coordinate (contained in $B(y, \rho)$). We now discuss how $T_{\zeta,w}^L$  is related to the large tube $T_{\theta,v}^{\ell}$. 
\begin{lemma}\label{lm: small wp2}For any $(\zeta, w, L)\in  \widetilde{\mathbb{T}}_{\theta,v, \ell}$, there holds 
	\[{\rm Angle}(L(\theta), L(\zeta))\lesssim \rho^{-1/2},\qquad {\rm Angle}(M(\theta), M(\zeta))\lesssim \rho^{-1/2},\]and 
	\[{\rm Dist}([T_{\theta,v}^{\ell}\cap B(y, 2\rho)]+2P^0_{\theta}, T^{L}_{\zeta,w})\lesssim R^{\delta}.\]
\end{lemma}
\begin{proof}
First, it is obvious to see that 
\[
{\rm Angle}(L(\theta), L(\zeta))\lesssim \rho^{-1/2},\qquad {\rm Angle}(M(\theta), M(\zeta))\lesssim \rho^{-1/2}
\]from the definition of $\widetilde{\mathbb{T}}_{\theta,v, \ell}$. It thus suffices to show the last inequality.

By definition, in the $x$ coordinate,
$$T^{L}_{\zeta,w} = P^{L}_{\zeta,w} + \rho L(\zeta) + \{y\}.$$	
(When the ball $B(y, \rho)$ is clear from the context, by abusing  notation, we use $T^L_{\zeta,w}$ to denote the set $P^L_{\zeta,w}+\rho L(\zeta)+\{y\}$, where $\rho L(\zeta):=\{t L(\zeta):\, 0\leq t\leq \rho\}$ and similarly for the $\rho L(\theta)$ below.) By Lemma~\ref{lm: small wp}, 
\[
{\rm Dist}(P^{L}_{\zeta, w}, \, P^{\ell}_{\theta,v}+P^{0}_{\theta} -\partial_{\xi} \psi_y (\xi_{\theta}))\lesssim R^{\delta}.
\]Moreover, one has $P^L_{\zeta,w} +\rho L(\zeta) \subset P^{\ell}_{\theta,v}+P^0_{\theta}+\rho L(\theta)$ because
${\rm Angle}(L(\zeta),L(\theta))\lesssim \rho^{-1/2}$. 

Since $|\xi_{\theta}|=1$, we have $\partial_{\xi}\psi_{y}(\xi_{\theta}) = y' + y_n \xi_{\theta}$ where $y':=(y_1, \dots, y_{n-1})$.  
So $${\rm Dist}(T^{L}_{\zeta,w}, \, 2P^0_{\theta} + P^{\ell}_{\theta,v} + \rho L(\theta) - \{y_n L(\theta)\})\lesssim R^\delta.$$

It suffices to show that $P^{\ell}_{\theta,v} +\rho L(\theta) - \{y_n L(\theta)\} \subset T^{\ell}_{\theta,v} \cap B(y, 2\rho)$, which is obviously the case. 
\end{proof}

\begin{remark}\label{rmk: small wp}
	Given a ball $B(y, \rho)$, in the paraboloid case treated in Section 7 of  \cite{Guth2}, many large wave packets $(\theta,v)$ might give rise to essentially the same set $\widetilde{\mathbb{T}}_{\theta,v}$ (which is the analog of our set $\widetilde{\mathbb{T}}_{\theta,v,\ell}$; see \cite[Page 30, (7.1)]{Guth2} for the exact definition). The reason is that for any $\theta_1, \theta_2 \subset \zeta$, if $B(y, \rho)\cap R^{\delta}T_{\theta_1, v_1}\cap R^{\delta}T_{\theta_2, v_2}\neq \emptyset$, then $B(y, 2\rho)\cap 2R^{\delta} T_{\theta_1, v_1} \cap 2R^{\delta}T_{\theta_2, v_2}$ contains a medium tube segment $T_{\rho}$ of length $\rho$ and radius $R^{1/2+\delta}$. And both $\widetilde{\mathbb{T}}_{\theta_j, v_j}$, $j=1,2$, consist of all the small wave packets $(\zeta,w)$ such that the small tube $T_{\zeta,w}\subset T_{\rho}$. 
	
	However, in the cone case, it is not true anymore that many $(\theta,v, \ell)$ always give rise to essentially the same set $\widetilde{\mathbb{T}}_{\theta,v, \ell}$. This is because $T^{\ell}_{\theta,v}$ is too thin in the mini direction $M(\theta)$. If $\theta_1, \theta_2\subset \zeta$ and $B(y, \rho)\cap R^{\delta}T^{\ell_1}_{\theta_1, v_1} \cap R^{\delta}T^{\ell_2}_{\theta_2, v_2}\neq \emptyset$, then it might happen that $B(y,2\rho)\cap 2R^{\delta}T^{\ell_1}_{\theta_1, v_1} \cap 2R^{\delta}T^{\ell_2}_{\theta_2, v_2}$ is contained in $2R^{\delta}T^{L}_{\zeta,w}$ for a single small wave packet $T^L_{\zeta,w}$. 
	
	In fact, a small wave packet $T^{L}_{\zeta,w}$ might belong to as many as $(\frac{R}{\rho})^{\frac{n-2}{2}}$ different $\widetilde{\mathbb{T}}_{\theta,v,\ell}$, which is about the number of disjoint $\theta\subset \zeta$. 
\end{remark}

Remark~\ref{rmk: small wp} suggests that it is difficult to use medium tube segments (of uniform length and radius) as a bridge to pass back and forth between large and small wave packets, unlike the situation with the paraboloid. Hence, in the following, we will only focus on grouping large and small wave packets into different sub-collections, which play the role of the ``medium tubes'' in the cone case. 

Here are the details. Let $\zeta_0$ be a sector on $2\bar{B}^{n-1}\setminus B^{n-1}$ of radius $\rho^{-1/2}$, and $v_0\in R^{1/2+\delta} \mathbb{Z}^{n-1}\cap B(0,\rho)$. We define the set $\widetilde{\mathbb{T}}_{\zeta_0, v_0}$ as follows: 
\[
\widetilde{\mathbb{T}}_{\zeta_0, v_0} := \left\{(\zeta, w, L): \, {\rm Dist}(\zeta_0, \zeta)\lesssim \rho^{-1/2}, \,P^{L}_{\zeta,w}\subset B(v_0, R^{1/2+2\delta}),\, L=1, \dots, \rho^{1/2}\right\},
\]where $P^{L}_{\zeta,w}$ is the essential support of the Fourier transform of $\widetilde{f}_{\zeta, w}^{L}$. The tube $T^{L}_{\zeta,w}$ in the $x$ coordinate is $T^{L}_{\zeta,w}=P^{L}_{\zeta,w}+\rho L(\zeta)+\{y\}$ where $\rho L(\zeta)$ is the line segment $\{ tL(\zeta): 0\leq t \leq \rho\}$. We also define the collection 
\[
\begin{split}
\mathbb{T}_{\zeta_0, v_0}(y):=& \Big\{  (\theta,v, \ell): \, {\rm Dist}(\theta,\zeta_0)\lesssim \rho^{-1/2},\,  T_{\theta,v}^{\ell} \cap B(y,\rho) \subset B(v_0, R^{1/2+2\delta})+\rho L(\zeta_0) +\{y\}, \\
&\qquad\qquad\qquad\qquad\qquad\qquad\qquad\qquad\qquad\qquad\qquad\qquad\qquad\qquad \ell =1, \dots, R^{1/2} \Big\}.
\end{split}
\]

For $v_0, v_0'$ satisfying ${\rm Dist}(v_0, v_0')\gtrsim R^{1/2+2\delta}$, one has that $\widetilde{\mathbb{T}}_{\zeta_0, v_0}$ and $\widetilde{\mathbb{T}}_{\zeta_0, v_0'}$ are essentially disjoint. So are $\mathbb{T}_{\zeta_0, v_0}(y) $ and $\mathbb{T}_{\zeta_0,v_0'}(y)$. In addition, the collections $\widetilde{\mathbb{T}}_{\zeta_0, v_0}$ and $\widetilde{\mathbb{T}}_{\zeta'_0, v_0}$ are essentially disjoint if ${\rm Dist}(\zeta_0,\zeta'_0)\gtrsim \rho^{-1/2}$. So are $\mathbb{T}_{\zeta_0, v_0}(y)$ and $\mathbb{T}_{\zeta_0', v_0}(y)$. Moreover, the collections $\widetilde{\mathbb{T}}_{\zeta_0, v_0}$ and $\mathbb{T}_{\zeta_0, v_0}(y)$ exhaust the set of all small wave packets $\{(\zeta,w,L)\}$ and the set of all large wave packets $\{(\theta,v,\ell)\}$ that intersect $B(y,\rho)$ respectively as $(\zeta_0,v_0)$ ranges over all possible choices.

Furthermore, for any $(\zeta_0,v_0)$, these two collections are naturally connected:
\[
\widetilde{\mathbb{T}}_{\zeta_0, v_0} = \underset{(\theta, v, \ell)\in \mathbb{T}_{\zeta_0, v_0}(y)}{\bigcup} \widetilde{\mathbb{T}}_{\theta, v, \ell}.
\]Therefore, applying Lemma~\ref{lm: small wp} and \ref{lm: small wp2}, one immediately obtains the following corollary. 
\begin{lemma}\label{lm: small wp g}
	If $g$ is concentrated on large wave packets in $\mathbb{T}_{\zeta_0, v_0}(y)$, then  $\widetilde{g} = e^{i\psi_y} g$ is concentrated on small wave packets in $\widetilde{\mathbb{T}}_{\zeta_0,v_0}$. On the other hand, if $\widetilde{g}$ is concentrated on small wave packets in $\widetilde{\mathbb{T}}_{\zeta_0, v_0}$, then inside $B(y, \rho)$, $g$ is concentrated on large wave packets in $\mathbb{T}_{\zeta_0, v_0}(y)$.
\end{lemma}

We need a few more notations before wrapping up the discussion on large and small wave packets. For a given ball $B(y, \rho)$ and any function $g$,   define the part of $g$ concentrated on large wave packets from $\mathbb{T}_{\zeta_0, v_0}(y)$ as $g_{\zeta_0, v_0}$:
\begin{equation}
g_{\zeta_0,v_0}:=\sum_{(\theta,v,\ell)\in \mathbb{T}_{\zeta_0, v_0}(y)} g_{\theta,v}^\ell,
\end{equation}
\begin{equation}
\widetilde{g}_{\zeta_0,v_0}:=\sum_{(\zeta,w,L)\in \widetilde{\mathbb{T}}_{\zeta_0, v_0}} \widetilde{g}_{\zeta,w}^L.
\end{equation}These give rise to the following decompositions of $g$ and $\widetilde{g}$ into wave packets that are grouped together by collections $\mathbb{T}_{\zeta_0,v_0}(y)$ and $\widetilde{\mathbb{T}}_{\zeta_0,v_0}$ respectively:
\[
g= \sum_{(\zeta_0,v_0)} g_{\zeta_0,v_0}+{\rm RapDec}(R)\|g\|_{L^2},\qquad \widetilde{g}= \sum_{(\zeta_0,v_0)} \widetilde{g}_{\zeta_0,v_0}+{\rm RapDec}(R)\|\widetilde{g}\|_{L^2},
\]
where the sums above are over all possible sectors $\zeta_0$ of width $\rho^{-1/2}$  partitioning the annulus, and all $v_0\in R^{\frac{1}{2}+\delta}\mathbb{Z}^{n-1}\cap B(0,\rho)$. 

Moreover, it is easy to see that both decompositions satisfy orthogonality:
\[
\|g\|_{L^2}^2\sim \sum_{(\zeta_0,v_0)} \|g_{\zeta_0,v_0}\|^2_{L^2},\qquad \|\widetilde{g}\|_{L^2}^2\sim \sum_{(\zeta_0,v_0)} \|\widetilde{g}_{\zeta_0,v_0}\|^2_{L^2}.
\]These decompositions will be used later in the transversal sub-case. The discussion on how to adjust $f$ into a wave packet decomposition inside a smaller ball $B(y,\rho)$ is complete.

\vspace{10pt}
Next, we will go back to the algebraic case and study its two sub-cases. Recall that we need to study the following situation: there is a function $g$ that is concentrated on wave packets in $\mathbb{T}_{\mathbf{Z}}$, and we would like to study $Eg$ restricted on a smaller ball $B(y, \rho)\subset B_R$. 


\subsubsection{The tangential sub-case}

In this part, suppose 
\[
\sum_{j}\|Ef_{j, {\rm tang}}\|^p_{BL^p_{k,A/2}(B_j)}\gtrsim \|Ef\|^p_{BL^p_{k,A}(B_R)}.
\]We would like to apply the induction hypothesis that the desired estimate holds for $m-1$, $\frac{A}{2}$ and $\rho$. Hence, in each $B_j$, we need to redo the wave packet decomposition of $f_{j, {\rm tang}}$ at the smaller scale $\rho$ and verify that the assumptions in Theorem \ref{theorem8.1} are satisfied, i.e. $f_{j,{\rm tang}}$ is concentrated on small wave packets that are $\rho^{-1/2+\delta_{m-1}}$-tangent to the variety $\mathbf{Y}$ in the ball $B_j$ of radius $\rho$.


%

Once we understood how to adjust $Ef_j$ into small wave packets in $B_j$ for the cone in the previous subsection, the verification of these properties is very similar to the paraboloid case (see \cite[Section 8.3]{Guth2}). We sketch the idea here. We know that $f_{j,{\rm tang}}$ is concentrated on wave packets $(\theta,v, \ell) \in \mathbb{T}_{j,{\rm tang}}$. 

To ease the notation, let $g=f_{j, {\rm tang}}$ and decompose 
$$\widetilde{g}=\sum_{\zeta,w, L} \widetilde{g}_{\zeta, w}^L  +{\rm RapDec}(R)\|f\|_{L^2}.$$ We would like to check that $\widetilde{g}$ is concentrated on wave packets  $(\zeta, w, L)$ tangential to $\mathbf{Y}$ in $B_j$ in the sense of Definition~\ref{tangent}. In other words, we would like to show that $\widetilde{g}$ is concentrated on wave packets $(\zeta, w, L)$ such that 
\begin{equation}\label{physical distance}
T_{\zeta, w}^L \subset N_{\rho^{1/2+\delta_{m-1}}}(\mathbf{Y})\cap B_j,
\end{equation}
and for any $x\in T_{\zeta,w}^L$,  and any $y\in \mathbf{Y}\cap B_j$ with $|x-y|\lesssim \rho^{1/2+\delta_{m-1}}$,
\begin{equation}\label{angle distance}
{\rm Angle}(L(\zeta), T_{y} \mathbf{Y})\lesssim \rho^{-1/2+\delta_{m-1}}.
\end{equation}

We know that $g=f_{j,{\rm tang}}$ is concentrated on wave packets $(\theta,v, \ell)\in \mathbb{T}_{j,{\rm tang}}$, which by definition obeys (\ref{eqn:verify1}) and (\ref{eqn:verify2}). These inequalities imply that $T_{\theta,v}^{\ell}\cap B_j$ lies in the desired neighborhood of $\mathbf{Y}\cap B_j$ and makes a small enough angle with $T_y(\mathbf{Y})$.  By Lemma~\ref{lm: small wp},   for any $(\theta,v, \ell)$, $(f^{\ell}_{\theta,v})^{\sim} $ is concentrated on wave packets $(\zeta, w, L)\in \widetilde{\mathbb{T}}_{\theta,v, \ell}$.  By the definition of $\widetilde{\mathbb{T}}_{\theta,v, \ell}$ and Lemma~\ref{lm: small wp2}, if $(\theta, v, \ell)\in \mathbb{T}_{j, {\rm tang}}$ and $(\zeta,w , L)\in \widetilde{\mathbb{T}}_{\theta,v, \ell}$, then $T^{L}_{\zeta,w }$ obeys (\ref{physical distance}) and (\ref{angle distance}).

We have thus verified the hypotheses of Theorem~\ref{theorem8.1} for $\widetilde{g}$ with the variety $\mathbf{Y}$ on the ball $B_j$, and so by induction on dimension, we get for each $j$ that
\[
\|Ef_{j,{\rm tang}}\|_{BL^p_{k,A/2}(B_j)}\leq C(K,\epsilon/2,m-1,D(\epsilon,D_{\mathbf{Z}}))\rho^{\epsilon/2}\rho^{\delta(\log{\bar{A}}-\log(A/2))}\rho^{-e+\frac{1}{2}}\|f_{j,{\rm tang}}\|_{L^2}
\]for all 
\[
2\leq p\leq p(m-1,k):=2\cdot\frac{m-1+k}{m-1+k-2}\qquad \text{with } e=\frac{1}{2}\left(\frac{1}{2}-\frac{1}{p} \right)(n+k).
\]Note that $p(m,k)<p(m-1,k)$, so the above estimate applies to all $p\in [2, p(m,k)]$. Summing over all the balls $B_j$ (with total number $\lesssim R^{O(\delta_{m-1})}$), one has
\[
\begin{split}
\|Ef\|_{BL^p_{k,A}(B_R)}\lesssim &R^{O(\delta_{m-1})} C(K,\epsilon/2,m-1,D(\epsilon,D_{\mathbf{Z}}))\rho^{\epsilon/2}\rho^{\delta(\log{\bar{A}}-\log(A/2))}\rho^{-e+\frac{1}{2}}\|f\|_{L^2}\\
\lesssim &R^{O(\delta_{m-1})}C(K,\epsilon/2,m-1,D(\epsilon,D_{\mathbf{Z}})) R^{\epsilon/2} R^{\delta(\log{\bar{A}}-\log A)}R^{-e+\frac{1}{2}}\|f\|_{L^2},
\end{split}
\]where the last step follows from the observation that $\rho^{-e+1/2}\leq R^{O(\delta_{m-1})}R^{-e+1/2}$ and 
\[\rho^{\delta(\log{\bar{A}}-\log(A/2))}\leq R^\delta R^{\delta(\log{\bar{A}}-\log A)},\]recalling that $\delta\ll \delta_{m-1}$.

Since $\delta_{m-1}\ll \epsilon$, one has $R^{O(\delta_{m-1})}R^{\epsilon/2}\lesssim R^{\epsilon}$. The induction thus closes if one chooses $C(K,\epsilon,m,D_{\mathbf{Z}})$ larger than $C(K,\epsilon/2,m-1,D(\epsilon,D_{\mathbf{Z}}))$. The discussion of the tangential sub-case is complete.

\subsubsection{The transversal sub-case}\label{sec: equi}

In the transversal case, our goal is to estimate 
\[
\sum_{j}\|Ef_{j, {\rm trans}}\|^p_{BL^p_{k,\frac{A}{2}}(B_j)},\]assuming that it dominates $\|Ef\|^p_{BL^p_{k,A}(B_R)}$. Our first claim is
\begin{equation}\label{eqn: trans1}
\sum_j\|f_{j,{\rm trans}}\|_{L^2}^2\lesssim \text{Poly}(D(\epsilon,D_{\mathbf{Z}})) \|f\|_{L^2}^2,
\end{equation} where $\text{Poly}(D(\epsilon,D_{\mathbf{Z}}))$ is a polynomial of $D(\epsilon,D_{\mathbf{Z}})$. Since $D(\epsilon,D_{\mathbf{Z}}) \lesssim_{\epsilon, m} 1$,  the constant  $\text{Poly}(D(\epsilon,D_{\mathbf{Z}}))\lesssim_{\epsilon, m} 1$. Inequality~\eqref{eqn: trans1}   will enable us to reduce the desired estimate to be inside each individual $B_j$.

To see (\ref{eqn: trans1}), one rewrites its left hand side as
\[
\sum_j\|f_{j,{\rm trans}}\|_{L^2}^2=\sum_{(\theta,v,\ell)}|\{j:\,(\theta,v,\ell)\in \mathbb{T}_{j,{\rm trans}}\}|\|f_{\theta,v}^\ell\|^2_{L^2}.
\]Then it suffices to show that $|\{j:\,(\theta,v,\ell)\in \mathbb{T}_{j,{\rm trans}}\}|\lesssim_{\epsilon,D_{\mathbf{Z}}}1$ for each $(\theta,v,\ell)$. According to \cite[Beginning of Section 8.4]{Guth2}, this is indeed the case. In fact, it is true even if one replaces the wave packet $T_{\theta,w}^\ell$ by a cylinder (with radius $r=R^{1/2+\delta_m}=\rho^{1/2+\delta_{m-1}}$ and the same central line as $T_{\theta,w}^\ell$ in the long direction). This is in particular a consequence of Lemma \ref{lem: trans} (with the choice $\alpha=\rho^{-1/2+\delta_{m-1}}$) and we omit the details.

Therefore, in the following, we would like to estimate $Ef_{j,{\rm trans}}$ in each ball $B_j$ and apply induction on the radius $R$. The induction hypothesis is: suppose $f$ is concentrated on wave packets from $T_{\mathbf{Z}}$, the collection of wave packets $T_{\theta,v}^\ell$ (at scale $R$) that are $\rho^{-1/2+\delta_m}$-tangent to the $m$-dimensional variety $\mathbf{Z}$ in $B_j$, 
then $$\|Ef\|_{BL^{p}_{k,A}(B_j)} \leq C(K,\epsilon, m,D_{\mathbf{Z}}) \rho^{\epsilon+ O(\delta)-e+\frac{1}{2}}\|f\|_{L^2}$$
where $2\leq p\leq p(m,k)$ and $e= \frac{1}{2}(\frac{1}{2}-\frac{1}{p})(n+k)$. 
There are two barriers preventing us from applying the induction hypothesis directly. 

First, $Ef$ is only known to be concentrated in the $R^{\frac{1}{2}+\delta_m}$-neighborhood of $\mathbf{Z}$, which is larger than the needed $\rho^{\frac{1}{2}+\delta_m}$-neighborhood. Therefore, one needs to decompose the $R^{\frac{1}{2}+\delta_m}$-neighborhood of $\mathbf{Z}$ into different layers of thickness $\rho^{\frac{1}{2}+\delta_m}$ so that each layer is a $\rho^{\frac{1}{2}+\delta_m}$-neighborhood of a translate $\mathbf{Z}_b$ of $\mathbf{Z}$. We also need to do a wave packet decomposition in $B_j$ at the scale $\rho$ similarly as in the tangential sub-case. Write $g=f_{j,{\rm trans}}$ and decompose 
\[
\widetilde{g}= \sum_{(\zeta,w,L)}\widetilde{g}_{\zeta,w}^{L} + \text{RapDec}(\rho) \|f\|_{L^2}.\]One needs to verify that each small wave packet $T_{\zeta,w}^L$ lies inside a unique layer $\mathbf{Z}_{b}$ and that $T_{\zeta,w}^{L}$ is $\rho^{-1/2+\delta_m}$-tangent to $\mathbf{Z}_{b}$. This is true and can be argued in the same way as in the paraboloid case: each small wave packet comes from some large wave packets that are even more tangent to $\mathbf{Z}$, so the small wave packet lies entirely in some layer $\mathbf{Z}_b$ and is also $\rho^{-1/2+\delta_m}$-tangent to $\mathbf{Z}_b$. The justification proceeds in the exact same way as the paraboloid case, which we will sketch later and refer the interested reader to \cite[Section 7, page 32-33]{Guth2} for more details.

Second, notice that $\rho^{\epsilon+ O(\delta)-e+\frac{1}{2}}$ is greater than $R^{\epsilon+ O(\delta)-e+\frac{1}{2}}$ for $p(n,k)\leq p\leq p(m,k)$. In order to obtain the correct (negative) power, one needs to find more structure between different layers. In the paraboloid case, the $L^2$-norms of $f$ on different layers turn out to be roughly the same, which is  referred to as the \emph{equidistribution} phenomenon. This is a key ingredient in the treatment of the corresponding case for the paraboloid in \cite{Guth2}. However, the argument there doesn't apply to the cone since our tubes are thinner and there are different mini directions existing for each wave packet. We solve this issue by showing that there still holds an analogous version of the equidistribution property for the cone, once a negligible part of $f$ is removed. This is one of the main novelties of our proof. In the following, we first establish the equidistribution property, then apply it to complete the proof of the transversal sub-case. \\ 

\noindent {\bf Transverse equidistribution estimates}

Intuitively, the property of equidistribution holds true because of the following heuristic:  when all the tubes are tangent to an $m$--dimensional low degree sub-variety $\mathbf{Z}$, the situation is similar to a $k$-broad restriction problem in $\mathbb{R}^m$. 

Given a point $\xi=(\xi_1, \dots , \xi_n)$ on the cone
\[
\mathcal{C}=\left\{\xi\in\mathbb{R}^n:\,\xi_1^2+\cdots+\xi_{n-1}^2=\xi_n^2,\, \xi_n>0,\, 1\leq\xi_j\leq 2,\, \forall 1\leq j\leq n-1\right\},
\]the normal direction $\mathbf{n}_\xi$ at $\xi$ is parallel to $(\xi_1, \dots, \xi_{n-1}, -\xi_n)$. Fix a ball $B$ of radius $R^{1/2+\delta_m}$. Let $V$ be the tangent space of $\mathbf{Z}$ at some point in $B\cap \mathbf{Z}$. Note that in hindsight, it does not matter which point we pick (because of Definition \ref{tangent} of tangent tubes).
Assume that $V$ is given by the equations $$\sum_{j=1}^n a_{i,j}x_j = b_i;\quad i= 1, \dots, n-m,$$ then the collection of all points $\xi$ on $\mathcal{C}$ such that the normal vector $\mathbf{n}_\xi$  of $\mathcal{C}$ at $\xi$ is parallel to $V$ lies in the vector space $V^+$, given by
$$\sum_{j=1}^{n-1} a_{i,j}\xi_j -a_{i,n}\xi_n=0;\quad i=1,\dots, n-m.$$

Recalling (\ref{tangentset}), define
\[
\mathbb{T}_{B,\mathbf{Z}}:=\left\{(\theta,v,\ell)\in\mathbb{T}_\mathbf{Z}:\,T_{\theta,v}^\ell\cap B\neq\emptyset\right\}.
\]For any function $h:\,2\bar{B}^{n-1}\setminus B^{n-1}\to \mathbb{C}$, let $h_B:=\sum_{(\theta,v,\ell)\in \mathbb{T}_{B,\mathbf{Z}}} h_{\theta,v}^\ell$. Define the lift of $h_B$ onto the cone as $H_B(\cdot):=h_B\circ \pi(\cdot)$, where $\pi$ denotes the projection from the cone $\mathcal{C}$ onto its first $(n-1)$ coordinates. Then, one observes that the support of $H_B$ lies inside $N_{R^{-1/2+\delta_m}}V^{+} \cap \mathcal{C}$. Indeed, $\text{supp}\,H_B$ lies inside $N_{R^{-1/2+\delta_m}}V^{+} $ by the definition of tangential wave packets, and $\text{supp}\,H_B$ lying in $\mathcal{C}$ is due to the definition of $H_B$. 

\begin{remark}\label{rmk: 5.7}
What does $N_{R^{-1/2+\delta_m}}V^{+} \cap \mathcal{C}$  look like? One special case is when $V^{+}$ is tangent to $\mathcal{C}$. As shown in the proof of Lemma~\ref{case} below, in this case $\dim V^{+}\cap \mathcal{C} =1$ and $N_{R^{-1/2+\delta_m}}V^{+} \cap \mathcal{C}$ is a $R^{-1/4+2\delta_m}$-neighborhood of  few radial line segments. In general, if $V^{+}$ is tangent to $\mathcal{C}$ up to an angle of $R^{-\delta_m}$ (``$K^{-2}$'' in Lemma~\ref{case} below), $N_{R^{-1/2+\delta_m}}V^{+} \cap \mathcal{C}$ is an $O(R^{-\delta_m})$-neighborhood of few radial line segments.
%

%
\end{remark}

%

\begin{lemma}\label{case}
Decompose $\mathbb{R}^n = V^+\oplus W$ so that $V^+\perp W$. Then either a) or b) is true:

a) $W$ and $V$ are transversal in the sense that $\text{Angle}(V, W)>K^{-2}$;

b) $\supp h_B$ is contained in the union of $O(1)$ many sectors $\tau_j$ in $2\bar{B}^{n-1}\setminus B^{n-1}$ of dimensions $1\times K^{-2}\times\cdots\times K^{-2}$. 
\end{lemma}
 \begin{proof}
 Let $\bar{\alpha}_i = (a_{i,1},\dots,a_{i,n-1})$ and  $\alpha_i = (\bar{\alpha}_i, -a_{i,n})$. Suppose there exists $w\in W$ such that $\text{Angle}(w, V)\leq K^{-2}$. Since $W\perp V^+$, one can write $w=\sum_{i=1}^{n-m}\lambda_i\alpha_i=:(\bar{w}, -w_n)$. Then by definition, it is straightforward to check that $(\bar{w},w_n)\in V^\perp$, which in particular implies that the angle between $w=(\bar{w}, -w_n)$ and $(\bar{w}, w_n)$ lies in the interval $\left[\frac{\pi}{2}-K^{-2},\frac{\pi}{2}+K^{-2}\right]$. Hence, the following hold true for $w$ and $\xi=(\bar{\xi}, \xi_n)\in\text{supp}\,H_B$:
 \begin{align*}
& \left|\frac{ |\bar{w}|^2 - w_n^2}{|\bar{w}|^2+w_n^2}\right|\lesssim K^{-2},\\
& |\bar{\xi}|^2-\xi_n^2=0,\\
& \frac{\left|\bar{\xi}\cdot\bar{w}-\xi_n w_n\right|}{|\xi||w|}\lesssim R^{-1/2+\delta_m}.
 \end{align*}
 
After renormalization so that $|w|=1$, the first inequality above shows that $\left|w_n^2- \frac{1}{2} \right|\leq C\cdot K^{-2}$. Combining the last two estimates together, we have $$\frac{|\bar{\xi}\cdot\bar{w}|}{|\bar{\xi}|\cdot|\bar{w}|}\geq\frac{|w_n\xi_n|}{|\bar{\xi}|\cdot|\bar{w}|}-C\cdot R^{-1/2+\delta_m}=\frac{|w_n|}{|\bar{w}|}-C\cdot R^{-1/2+\delta_m}\geq 1-C\cdot K^{-2}.$$
Thus the support of $h_B$ must lie in an $O(K^{-2})$-angular neighborhood of $\bar{w}$. In particular, $\supp h_B$ lies in an $O(K^{-2})$-angular region in $2\bar{B}^{n-1}\setminus B^{n-1}$, hence case b) is true.
 \end{proof}
For a fixed variety $\mathbf{Z}$, whether case a) or b) holds true depends only on the vector space $V$, in other words, only on the ball $B$. If we are in case b),  by the definition of the $BL^p_k$ norm, since $\text{supp} h_B$ is contained in the union of  $O(1)$ sectors  and we have chosen  $1\ll A \ll K$, 
 for all $k\geq 2$ there always holds
\begin{equation}\label{nonessential}
\|Eh_B\|_{BL^p_{k,A}(B)}^p=\mu_{Eh_B}(B)=\text{RapDec}(R) \|h_B\|_{L^2}^p.
\end{equation}On the other hand, if we are in case a), the following lemma, adapted from the paraboloid case (Lemma 6.2 of \cite{Guth2}), says that the $L^2$ norm of $Eh_B$ is equidistributed in $B$ along directions transverse to $V$. 
 \begin{lemma}\label{equidistribution physical space}
Let $h_B=\sum_{(\theta, v,\ell)\in \mathbb{T}_{B, \mathbf{Z}} }h_{\theta,v}^\ell$ and $\mathbf{Z}$ be defined as in Theorem~\ref{theorem8.1}. Suppose that $B$ is a ball of radius $R^{1/2+\delta_m}$ in $B_R\subset\mathbb{R}^n$, and satisfies case a) of Lemma \ref{case}. Then for any $\rho \leq R$,
$$ \int_{B\cap N_{\rho^{1/2+\delta_m}}(\mathbf{Z})}|Eh_B|^2 \lesssim R^{O(\delta_m)}\big( \frac{R^{1/2}}{\rho^{1/2}}\big)^{-(n-m)}\int_{2B}|Eh_B|^2 +\text{RapDec}(R)\|h_B\|_{L^2}^2.$$
\end{lemma}Note that the angle condition in case a) of Lemma \ref{case} is used in the analog of Lemma 6.5 of \cite{Guth2}, which is a key step in the proof of the above lemma.

	\begin{proof}	
	 Recall that $V$ is the tangent space of $\mathbf{Z}$ at some point in $B\cap \mathbf{Z}$, hence $$\mathbb{T}_{B, \mathbf{Z}}\subset \mathbb{T}_{B, V}:=\{ (\theta,v, \ell):\, T^{\ell}_{\theta,v} \cap B\neq \emptyset \text{ and } \text{Angle}(L(\theta), V)\lesssim R^{-\frac{1}{2}+\delta_{m}}\}.$$

According to the discussion above Remark~\ref{rmk: 5.7}, for all $(\theta,v,\ell)\in \mathbb{T}_{B,\mathbf{Z}}$, $(Eh^{\ell}_{\theta,v})^{\wedge}$ is supported in $N_{R^{-1/2+\delta_m}} V^{+} \cap \mathcal{C}$. For any $(n-m)$--plane $\Pi$ parallel to $W$ passing through $B$, if we view the restriction of $Eh_B$ on $\Pi$ as a function $G: \Pi\rightarrow \mathbb{C}$, then its Fourier transform is supported in a ball of radius $\lesssim R^{-1/2+\delta_m}$ because $V^{+}\perp W$. Therefore, by Lemma 6.4 in \cite{Guth2}, 
\begin{equation}\label{projection}
\int_{\Pi\cap B(x_0, \rho^{1/2+2\delta_m})} |Eh_B|^2\lesssim (\frac{R^{1/2-2\delta_m}}{\rho^{1/2+2\delta_m}})^{-\text{Dim}W} \int_{\Pi}  W_{B(x_0, R^{1/2-2\delta_m})} |Eh_B|^2,
\end{equation}
where $x_0$ is any point and $W_{B(x_0, R^{1/2-\delta_m})}$ is a weight that is equal to $1$ on $B(x_0, R^{1/2-2\delta_m})$ and rapidly decaying outside of it. 
Since $\text{Angle}(V, W) > K^{-2}$, we have for some $x_0\in B$ that
\begin{equation}\label{key}
\Pi \cap N_{\rho^{1/2+\delta_m}}(\mathbf{Z}) \cap B \subset \Pi \cap B(x_0, \rho^{1/2+2\delta_m}).
\end{equation}

Therefore,
\begin{equation}
\begin{split}
\int_{\Pi \cap N_{\rho^{1/2+\delta_m}}(\mathbf{Z}) \cap B} |Eh_B|^2\leq &\int_{\Pi\cap B(x_0, \rho^{1/2+2\delta_m})} |Eh_B|^2\\
\lesssim &R^{O(\delta_m)} (\frac{R^{1/2}}{\rho^{1/2}})^{-(n-m)} \int_{\Pi}  W_{B} |Eh_B|^2.
\end{split}
\end{equation}Note that if $x\in \Pi\setminus 2B$, $|Eh_B(x)|\leq \text{RapDec}(R)\|h_B\|_{L^2}$, which implies
\[
\int_{\Pi}  W_{B} |Eh_B|^2\leq \int_{\Pi\cap 2B}  W_{B} |Eh_B|^2 + \text{RapDec}(R)\|h_B\|_{L^2}^2.
\]Hence, by integrating over all $\Pi$ that are parallel to $W$ and passing through $B$, one obtains the desired estimate. 

\end{proof}

\begin{remark}
In the proof above, one can see that inequality (\ref{projection}) and  (\ref{key}) are the key estimates for the derivation of the transverse equidistribution of $Eh_B$. Note that inequality (\ref{projection}) is in fact general and stated as Lemma 6.4 in \cite{Guth2}. In the paraboloid case,  the angle condition implying (\ref{key}) always holds. However,  in the cone case, this is not always true, which is why we need to rule out the case b)  in Lemma~\ref{case}.
\end{remark}

The key property we are going to demonstrate is: inside each ball $B_j$ of radius $\rho$, the $L^2$ norm of the part of the function $f_{j,{\rm trans}}$ restricted in case a) of Lemma \ref{case} is equidistributed along the direction of a fixed vector $b$, the precise statement of which is postponed to Lemma \ref{precise} below. Unlike the paraboloid case, we do not have such equidistribution for the entire $f_{j,{\rm trans}}$, however, (\ref{nonessential}) ensures that the leftover part of $f_{j,{\rm trans}}$ is nonessential as it makes negligible contribution. 

Fix $B_j=B(y,\rho)$ and again write $g=f_{j, {\rm trans}}$ for short. Cover $B_j$ by balls $B$ of radius $R^{1/2+\delta_m}$ and partition $N_{R^{1/2+\delta_m}}(\mathbf{Z}) \cap B_{j}= X_a\cup  X_b$, where $X_a$ (resp. $X_b$) is the union of balls $B$ in case a) (resp. case b)) as defined in Lemma \ref{case}. 

Recall from Section \ref{subsec: largesmall} the definitions of collections $\mathbb{T}_{\zeta_0,v_0}(y)$ of large wave packets at scale $R$, $\widetilde{\mathbb{T}}_{\zeta_0,v_0}$ of small wave packets at scale $\rho$, and the notations $g_{\zeta_0,v_0}, \widetilde{g}_{\zeta_0,v_0}$, where $\zeta_0$ is a sector in $2\bar{B}^{n-1}\setminus B^{n-1}$ of radius $\rho^{-1/2}$ and $v_0\in R^{\frac{1+\delta}{2}} \mathbb{Z}^{n-1}\cap B(0,\rho)$.  We define $g_{{\rm ess}}$, the essential part of $f_{j,{\rm trans}}$ that is restricted in case a), as follows:
\[
g_{{\rm ess}}=\sum_{(\zeta_0,v_0)\in \mathbb{T}_{{\rm ess}}}g_{\zeta_0,v_0}=g-\sum_{(\zeta_0,v_0)\in \mathbb{T}_{{\rm tail}}}g_{\zeta_0,v_0},
\]where 
\[
\mathbb{T}_{{\rm ess}}:=\{(\zeta_0,v_0):\, \exists (\theta,v,\ell)\in \mathbb{T}_{\zeta_0,v_0}(y)\, \text{ s.t. } T^\ell_{\theta,v}\cap X_a\neq\emptyset\},
\]
\[
\mathbb{T}_{{\rm tail}}:=\{(\zeta_0,v_0):\, \forall (\theta,v,\ell)\in \mathbb{T}_{\zeta_0,v_0}(y),\,\, T^\ell_{\theta,v}\cap X_a=\emptyset\}.
\]Note that we only consider those $\mathbb{T}_{\zeta_0,v_0}$ that contains some large wave packet intersecting $B_j$, so $\mathbb{T}_{{\rm ess}}$ and $\mathbb{T}_{{\rm tail}}$ above form a partition of all $(\zeta_0,v_0)$ that matter to $B_j$.

\begin{remark}\label{rmk: B}
Another important observation is: for any given $(\zeta_0,v_0)$, if there exists $(\theta,v,\ell)\in \mathbb{T}_{\zeta_0,v_0}(y)$ and $R^{1/2+\delta_m}$-ball $B$ such that $T_{\theta,v}^\ell\cap B\cap B_{j}\neq\emptyset$, then for all $(\theta',v',\ell')\in \mathbb{T}_{\zeta_0,v_0}(y)$, one has $T_{\theta',v'}^{\ell'}\cap 2B\neq \emptyset$. 
This is because the union of all $T^{\ell}_{\theta,v}\cap B_j$ over $(\theta,v, \ell) \in \mathbb{T}_{\zeta_0, v_0}(y)$ is a short tube of length $\rho$ and radius $\sim R^{1/2 +2\delta}$. So when $B$ intersects this short tube, all $T^{\ell}_{\theta,v}$ automatically pass through $2B$. 
\end{remark}

We now reduce the estimate of $g=f_{j,{\rm trans}}$ to its essential part. By the triangle inequality Lemma \ref{lem: tri}, one has

\[
\begin{split}
\|Eg\|_{BL^p_{k,A}(B_j)}\leq&\|Eg_{\text{ess}}\|_{BL^p_{k,A/2}(B_j)}+\left\|E\big(\sum_{(\zeta_0,v_0)\in \mathbb{T}_{{\rm tail}}}g_{\zeta_0,v_0}\big)\right\|_{BL^p_{k,A/2}(B_j)}\\
\leq&\|Eg_{\text{ess}}\|_{BL^p_{k,A/2}(B_j)}+\left\|E\big(\sum_{(\zeta_0,v_0)\in \mathbb{T}_{{\rm tail}}}g_{\zeta_0,v_0}\big)\right\|_{BL^p_{k,A/2}(X_b)}+\text{RapDec}(R)\|f\|_{L^2}\\
=&\|Eg_{\text{ess}}\|_{BL^p_{k,A/2}(B_j)}+\text{RapDec}(R)\|f\|_{L^2}.
\end{split}
\]In the above, the second step is a consequence of the definition of $\mathbb{T}_{{\rm tail}}$, and the last step follows from (\ref{nonessential}). It thus suffices to study $g_{\text{ess}}$ from this point on. 

Now, we would like to choose a direction, given by a vector $b\in \mathbb{R}^{n-m}$ with $|b|\leq R^{1/2+\delta_m}$, and decompose the $R^{\frac{1}{2}+\delta_m}$-neighborhood of $\mathbf{Z}$ into layers of thickness $\rho^{\frac{1}{2}+\delta_m}$ along $b$. 

Fix a $(\zeta_0,v_0)\in \mathbb{T}_{{\rm ess}}$ and a $R^{1/2+\delta_m}$-ball $B\subset T_{\theta,v}^\ell\cap X_a$ for some $(\theta,v,\ell)\in \mathbb{T}_{\zeta_0,v_0}(y)$, then Lemma \ref{equidistribution physical space} applies to $g_{\zeta_0,v_0}$ and $B$. 

Recall that the ball $B$ determines locally a tangent space to $\mathbf{Z}$, denoted by $V$, and a vector space $V^+$ that contains points on the cone with normal direction parallel to $V$. Since $B\subset X_a$, according to Lemma \ref{case}, one has ${\rm Angle}(V,W)>K^{-2}$, where $W$ is the orthogonal complement of $V^+$ in $\mathbb{R}^n$. Choose any $b\in W$ with $|b|\leq R^{1/2+\delta_m}$, then the direction $b$ is transversal to $T_x\mathbf{Z}$ for all $x\in B\cap \mathbf{Z}$. In fact, by a simple reduction, one can assume without loss of generality that $b$ is transversal (by an angle at least  $K^{-2}$) to $T_{x}\mathbf{Z}$ for all points $x\in \mathbf{Z}\cap B_j$. 

Indeed, let $\Lambda$ be a $K^{-3}$-net of all directions in $\mathbb{R}^n$, then there are $O(K^{3(n-1)})$ many points (directions) in $\Lambda$. Decompose $N_{R^{1/2+\delta_m}}(\mathbf{Z}) \cap 2B_j= \bigcup U_s$ into $O(K^{3(n-1)})$ many disjoint parts, such that for each $x\in U_s\cap \mathbf{Z}$, the normal direction of $T_x\mathbf{Z}$ is $O(K^{-3})$ close to a point in $\Lambda$. The disjointness of $\left\{U_s\right\}_s$ and triangle inequality imply that 
\[
\|Ef\|_{BL^p_{k,A}(B_j)}^p =\sum_{s} \|Ef\|_{BL^p_{k,A}(U_s)}^p \leq \sum_{s}\|Ef_s\|_{BL^p_{k,A/2}(U_s)} + \text{RapDec}(R)\|f\|_{L^2}^p.
\]Here $Ef_{s} :=\sum_{T_{\theta, v}^\ell\cap U_s \neq \emptyset} Ef_{\theta,v}^\ell$, hence there is rapid decay of $|Ef-Ef_{s}|$ on $U_s$. It thus suffices to study each $Ef_s$ as there are only $O(K^{3(n-1)})\ll R^\epsilon$ many of them in total. 

After this reduction, in the following, independently of the ball $B$, the choice of the vector $b$ will be fixed, as it is transversal to the tangent plane $T_x\mathbf{Z}$ for all $x\in \mathbf{Z}\cap X_a$. Our goal is to show that the $L^2$ norm of $g_{{\rm ess}}$ is equidistributed along the direction $b$ in $N_{R^{1/2+\delta_m}}(\mathbf{Z})\cap B_j$.


Note that in the paraboloid setting dealt with in \cite{Guth2}, one can choose $b$'s freely in each $B$, since equidistribution in the physical space (Lemma 6.2 of \cite{Guth2}, the analog of our Lemma \ref{equidistribution physical space}) holds true on each $B$. We unfortunately do not have the luxury with the cone. In fact, it can be as bad that only one $B$ here has equidistribution. A key observation is that this is already good enough for us. Essentially speaking, the equidistribution of $Eg_{\text{ess}}$ in the physical space concluded in Lemma \ref{equidistribution physical space} will give rise to that of $\|g_{\text{ess}}\|_{L^2}$ in the frequency space, and after breaking $g_{\text{ess}}$ down into the orthogonal pieces $\{g_{\zeta_0,v_0}\}$, the behavior of $Eg_{\zeta_0,v_0}$ inside $B$ will control $\|g_{\zeta_0,v_0}\|_{L^2}$. We state this last observation as the following lemma, which is borrowed from the paraboloid case: Lemma 3.4 of \cite{Guth2}. Being a direct corollary of orthogonality of wave packets and Plancherel, it works in the cone case equally well. 

\begin{lemma}\label{orthog}
Suppose that $h$ is a function concentrated on a set of wave packets $\mathbb{T}$ and that for every $T^\ell_{\theta,v}\in\mathbb{T}$, $T_{\theta,v}^\ell\cap B_r(z)\neq\emptyset$ for some radius $r\geq R^{1/2+\delta_m}$. Then
\[
\|Eh\|_{L^2(B_{10r}(z))}^2\sim r\|h\|_{L^2}^2.
\]
\end{lemma}

This, together with Remark \ref{rmk: B}, immediately implies that for any $B\subset X_a$ such that $B\cap T_{\theta,v}^\ell\neq\emptyset$, where $(\theta,v,\ell)\in \mathbb{T}_{\zeta_0,v_0}(y)$ for some $(\zeta_0,v_0)\in \mathbb{T}_{{\rm ess}}$, there holds
\begin{equation}\label{wholepiece}
\|g_{\zeta_0,v_0}\|_{L^2}^2\sim R^{-1/2-\delta_m}\|Eg_{\zeta_0,v_0}\|_{L^2(40B)}^2.
\end{equation}

Along the direction of $b$, we decompose $N_{R^{1/2+\delta_m}}(Z)\cap B_j$ into layers of thickness $\sim \rho^{1/2+\delta_m}$. More precisely, choose a set of vectors $\mathcal{B}=\{b\}$ with $|b|\leq R^{1/2+\delta_m}$ such that $\{N_{\rho^{1/2+\delta_m}}(\mathbf{Z}+b)\cap B_j\}$ form a disjoint union of $N_{R^{1/2+\delta_m}}(\mathbf{Z})\cap B_j$. Since our goal is to induct on the radius, we now look at the small wave packet decomposition of $g$ in $B_j$ at scale $\rho$ and study how the small wave packets are distributed among different layers.

Decompose 
\[
\widetilde{g}= \sum_{(\zeta,w,L)}\widetilde{g}_{\zeta,w}^{L} + \text{RapDec}(\rho) \|f\|_{L^2}.\]Observe that for any $(\zeta,w,L)$ in the wave packet decomposition in $B_j$, if $T_{\zeta,w}^L$ intersects $N_{\rho^{1/2+\delta_m}}(\mathbf{Z}+b)\cap B_j$ for some $b\in\mathcal{B}$, then according to Lemma \ref{lm: small wp2}, $T_{\zeta,w}^L$ is contained in $N_{2\rho^{1/2+\delta_m}}(\mathbf{Z}+b)\cap B_j$ and moreover $T_{\zeta,w}^L$ is $2\rho^{-1/2+\delta_m}$-tangent to $\mathbf{Z}+b$ in $B_j$. Define
\begin{equation}\label{eq: deffb}
\widetilde{\mathbb{T}}_{\mathbf{Z}+b}:=\left\{(\zeta,w,L):\, T_{\zeta,w}^L\,\,\text{is } 2\rho^{-1/2+\delta_m}\text{-tangent to}\,\,\mathbf{Z}+b\,\,\text{in}\,\, B_j\right\},\quad \widetilde{g}_b:=\sum_{(\zeta,w,L)\in\widetilde{\mathbb{T}}_{\mathbf{Z}+b}}\widetilde{g}_{\zeta,w}^L.
\end{equation}Then one has
\[
\widetilde{g}_{{\rm ess}, b}=\sum_{(\zeta_0,v_0)\in \mathbb{T}_{{\rm ess}}}\sum_{(\zeta,w,L)\in \widetilde{\mathbb{T}}_{\zeta_0,v_0}\cap \widetilde{\mathbb{T}}_{\mathbf{Z}+b}}\widetilde{g}_{\zeta,w}^L.
\]For each $b\in\mathcal{B}$, $\widetilde{g}_{{\rm ess}, b}$ is concentrated on wave packets tangent to $\mathbf{Z}+b$ in $B_j$, hence the induction hypothesis at scale $\rho$ applies to $\widetilde{g}_{{\rm ess},b}$. 

The transverse equidistribution property as follows is the main estimate of this subsection.
\begin{lemma}\label{precise}
Let $g_{\text{ess}}$ and $b\in \mathcal{B}$ be defined as above, then
$$\|\widetilde{g}_{\text{ess}, b}\|_{L^2}^2 \leq R^{O(\delta_m)}\big(\frac{R^{1/2}}{\rho^{1/2}}\big)^{-(n-m)} \|g_{\text{ess}}\|_{L^2}^2.$$
\end{lemma}

\begin{proof}
For every $(\zeta_0,v_0)\in\mathbb{T}_{{\rm ess}}$, define
\[
\widetilde{g}_{\zeta_0,v_0,b}=\sum_{(\zeta,w,L)\in \widetilde{\mathbb{T}}_{\zeta_0,v_0}\cap \widetilde{\mathbb{T}}_{\mathbf{Z}+b}}\widetilde{g}_{\zeta,w}^L.
\]According to Lemma \ref{lm: small wp g}, $\widetilde{g}_{\zeta_0,v_0,b}$ is concentrated on large wave packets $T_{\theta,v}^\ell$ in $\mathbb{T}_{\zeta_0,v_0}(y)$. Moreover, according to Remark \ref{rmk: B}, there exists some $R^{1/2+\delta_m}$-ball $B\subset X_a$ such that $2B$ intersects all $T_{\theta,v}^\ell\in \mathbb{T}_{\zeta_0,v_0}(y)$. Since modulation doesn't change the $L^2$ norm, applying Lemma \ref{orthog} to $\widetilde{g}_{\zeta_0,v_0,b}$ and $2B$, one obtains
\begin{equation}\label{singlepiece}
\|\widetilde{g}_{\zeta_0,v_0,b}\|_{L^2}^2 \sim R^{-1/2-\delta_m}\|Eg_{\zeta_0,v_0}\|^2_{L^2(20B\cap N_{\rho^{1/2+\delta_m}(\mathbf{Z}+b)})}.
\end{equation}

Recall also from (\ref{wholepiece}) that
\[
\|g_{\zeta_0,v_0}\|_{L^2}^2\sim R^{-1/2-\delta_m}\|Eg_{\zeta_0,v_0}\|_{L^2(40B)}^2.
\]Combining (\ref{singlepiece}), Lemma \ref{equidistribution physical space}, and then (\ref{wholepiece}), we obtain the equidistribution for each $g_{\zeta_0,v_0}$:
\begin{equation}\label{equih}
\|\widetilde{g}_{\zeta_0,v_0,b}\|_{L^2}^2 \leq R^{O(\delta_m)}\big(\frac{R^{1/2}}{\rho^{1/2}}\big)^{-(n-m)}\|g_{\zeta_0,v_0}\|_{L^2}^2.
\end{equation}Note that strictly speaking, Lemma \ref{equidistribution physical space} only applies to $B$ instead of $20B$. However, it is easy to see that the same result holds true if we chose to use a constant dilation of $B$ from the beginning. We omit this technicality.

By orthogonality, one has $$\|g_{\text{ess}}\|_{L^2}^2 \sim \sum_{(\zeta_0,v_0)\in \mathbb{T}_{{\rm ess}}}\|g_{\zeta_0,v_0}\|_{L^2}^2,$$ and
$$\|\widetilde{g}_{\text{ess},b}\|_{L^2}^2 \sim \sum_{(\zeta_0,v_0)\in \mathbb{T}_{{\rm ess}}}\|\widetilde{g}_{\zeta_0,v_0,b}\|_{L^2}^2.$$
Hence the desired estimate follows immediately from (\ref{equih}).
\end{proof}

Now, we can use Lemma \ref{precise} to complete the estimate of the transversal sub-case. This part of the argument, once the above version of the equidistribution estimate is proved, proceeds in the exact same way as the paraboloid case. We provide only a sketch in the following while refer the interested reader to \cite[Section 8.4, page 42-43]{Guth2} for details. 

Recall that it suffices to estimate
\[
\sum_{j}\|Eg_{{\rm ess}}\|^p_{BL^p_{k,\frac{A}{2}}(B_j)}.
\]Note that $g_{{\rm ess}}$ also implicitly depends on $j$.

First, one has for each $B_j$ that
\[
\|Eg_{{\rm ess}}\|^{p}_{BL^{p}_{k, A/2}(B_j)} \lesssim (\log R) \sum_{b\in \mathcal{B}} \|Ef_{j,{\rm trans},b}^{\rm ess}\|^{p}_{BL^{p}_{k, A/2}(B_j)},
\]where $f_{j,{\rm trans},b}^{\rm ess}$ is defined so that $(f_{j,{\rm trans},b}^{\rm ess})^\sim=\widetilde{g}_{{\rm ess}, b}$, i.e. $f_{j,{\rm trans},b}^{\rm ess}=e^{-i\psi_y(\xi)}\widetilde{g}_{{\rm ess}, b}$. Hence, one has
\begin{equation}\label{eqn: last1}
\|Ef\|_{BL^{p}_{k,A}(B_R)}^p\lesssim (\log R )\sum_{j}\sum_{b\in \mathcal{B}} \|Ef_{j,{\rm trans},b}^{\rm ess}\|^{p}_{BL^{p}_{k, A/2}(B_j)}.
\end{equation}

Second, for each $(\zeta_0,v_0)\in \mathbb{T}_{{\rm ess}}$ and each $R^{1/2+\delta_m}$-ball $B\subset X_a$ that intersects all $T_{\theta,v}^\ell$ in $\mathbb{T}_{\zeta_0,v_0}(y)$, the sets $B\cap N_{\rho^{1/2+\delta_m}}(\mathbf{Z}+b)$ for different $b\in \mathcal{B}$ are disjoint. Hence, by (\ref{wholepiece}) and (\ref{singlepiece}), one has 
\[
\sum_{b\in\mathcal{B}}\|\widetilde{g}_{\zeta_0,v_0,b}\|_{L^2}^2\lesssim \|g_{\zeta_0,v_0}\|_{L^2}^2.
\]Since $\widetilde{g}_{{\rm ess}, b}=\sum_{(\zeta_0,v_0)\in \mathbb{T}_{{\rm ess}}} \widetilde{g}_{\zeta_0,v_0,b}$ is an orthogonal decomposition, and so is the decomposition $g_{{\rm ess}}=\sum_{(\zeta_0,v_0)\in \mathbb{T}_{{\rm ess}}}g_{\zeta_0,v_0}$, one has the estimate
\begin{equation}\label{eqn: last2}
\sum_{j}\sum_{b\in \mathcal{B}}\|f_{j,{\rm trans},b}^{{\rm ess}}\|_{L^2}^2=\sum_j\sum_{b\in \mathcal{B}}\|\widetilde{g}_{{\rm ess}, b}\|_{L^2}^2\lesssim  \text{Poly}(D_{\mathbf{Z}})\|g_{{\rm ess}}\|_{L^2}^2 \lesssim_{\epsilon, m}\|g_{{\rm ess}}\|_{L^2}^2  .
\end{equation}


Moreover, by the equidistribution estimate Lemma \ref{precise}, there holds
\begin{equation} \label{trans-equi}
\max_{b\in \mathcal{B}} \|f_{j,{\rm trans},b}^{\rm ess}\|_{L^2}^2
\leq R^{O(\delta_m)} \left(\frac{R^{1/2}}{\rho^{1/2}}\right)^{-(n-m)} \|g_{ess}\|_{L^2}^2.
\end{equation}

By the inductive hypothesis, one has for each $B_j$ that
\[
\begin{split}
\|Ef_{j,{\rm trans}, b}^{\rm ess}\|_{BL^p_{k,A/2}(B_j)}\leq &C(K,\epsilon, m, D_{\mathbf{Z}})\rho^\epsilon \rho^{\delta(\long\bar{A}-\log(A/2))}\rho^{-e+\frac{1}{2}}\|f_{j,{\rm trans},b}^{\rm ess}\|_{L^2}\\
\leq &C(K,\epsilon, m, D_{\mathbf{Z}}) R^\delta \rho^\epsilon R^{\delta(\log\bar{A}-\log A)}\rho^{-e+\frac{1}{2}}\|f_{j,{\rm trans},b}^{\rm ess}\|_{L^2}.
\end{split}
\]Combing the estimates (\ref{eqn: last1}), (\ref{eqn: last2}) and (\ref{trans-equi}) together, one has
\[
\begin{split}
\|Ef\|_{BL^{p}_{k,A}(B_R)}^p\lesssim &(\log R )\sum_{j}\sum_{b\in \mathcal{B}} \|Ef_{j,{\rm trans},b}^{\rm ess}\|^{p}_{BL^{p}_{k, A/2}(B_j)}\\
\lesssim & R^{O(\delta)}\Big( C(K,\epsilon, m, D_{\mathbf{Z}}) \rho^\epsilon R^{\delta(\log\bar{A}-\log A)}\rho^{-e+\frac{1}{2}} \Big)^p \sum_{j,b}\|f_{j,{\rm trans},b}^{\rm ess}\|_{L^2}^p\\
\lesssim & R^{O(\delta_m)} \Big( C(K,\epsilon, m, D_{\mathbf{Z}}) \rho^\epsilon R^{\delta(\log\bar{A}-\log A)}\rho^{-e+\frac{1}{2}} \Big)^p\left(\frac{R^{1/2}}{\rho^{1/2}}\right)^{-(n-m)(\frac{p}{2}-1)} \|f\|_{L^2}^p.
\end{split}
\]
When $p=p(m,k)$ as defined in Theorem \ref{theorem8.1}, there holds
\[
(\rho^{-e+\frac12})^p\left(\frac{R^{1/2}}{\rho^{1/2}}\right)^{-(n-m)(\frac{p}{2}-1)}=(R^{-e+\frac12})^p,
\]hence
\[
\|Ef\|_{BL^{p}_{k,A}(B_R)}^p\leq C_{\epsilon,D_{\mathbf{Z}}}R^{O(\delta_m)}(R/\rho)^{-\epsilon}\left(C(K,\epsilon,m,D_{\mathbf{Z}})R^\epsilon R^{\delta(\log\bar{A}-\log A)}R^{-e+\frac{1}{2}}\right)^p\|f\|_{L^2}^p.
\]Note that $R/\rho=R^{O(\delta_{m-1})}$. By choosing the parameters so that $\delta_m\ll \epsilon \delta_{m-1}$, one can have $(R/\rho)^{-\epsilon}$ dominate the other terms, thus $C_{\epsilon,D_{\mathbf{Z}}}R^{O(\delta_m)}(R/\rho)^{-\epsilon}\leq 1$. The induction in the transverse sub-case is closed, and we have completed the proof of Theorem \ref{theorem8.1}.

\section{$k$-broad estimate implies linear restriction}\label{linearmain}
\subsection{$L^p\to L^p$ restriction}\label{prescale}
In this subsection, we demonstrate how Theorem \ref{kBroadThm}, the $k$-broad estimate, implies the main Theorem \ref{main}, the linear cone restriction estimate. The first ingredient of the argument is a decoupling inequality for the cone derived by Bourgain and Demeter \cite{BD}, and the second one is an induction on scales argument. The main difference of the proof from the paraboloid case lies in the second step, where a Lorentz rescaling is applied to the cone.

More precisely, we are going to prove for any $R>1$ and $p\leq q\leq\infty$ that
\begin{equation}\label{linear}
\|Ef\|_{L^p(B_R)}\lesssim_\epsilon R^\epsilon\|f\|_{L^q},
\end{equation}whenever there holds the $k$-broad estimate
\[
\|Ef\|_{BL^p_{k,A}(B_R)}\lesssim_{K,\epsilon}R^\epsilon\|f\|_{L^q}
\]for $p$ in the range
\[
p(k,n)< p\leq\frac{2n}{n-2},\quad p(k,n):=\begin{cases} 2\cdot\frac{n-1}{n-2}&\text{if}\,\,2\leq k\leq 3,\\ 2\cdot\frac{2n-k+1}{2n-k-1}&\text{if}\,\, k>3. \end{cases}
\]
The upper bound of the range for $p$ comes from the requirement in the decoupling theorem below. Note that the lower bound $p(k,n)$ is different from the critical index $\bar{p}(k,n)=2\cdot\frac{n+k}{n+k-2}$ in Theorem \ref{kBroadThm}. We claim that (\ref{linear}) (with $p$ in the range above) implies Theorem \ref{main} immediately. Indeed, taking $k=\frac{n+1}{2}$ when $n$ is odd and $k=\frac{n}{2}+1$ when $n$ is even, $\max(p(k,n),\bar{p}(k,n))$ gives the lower bound for $p$ in Theorem \ref{main}. Then, Theorem \ref{main} follows by interpolating with the trivial $L^\infty$ bound of $E$ and $\epsilon$-removal (\cite{Tao1}).

We now begin the proof of (\ref{linear}). The $k$-broad estimate assumption says that
\[
\sum_{B_{K^2}\subset B_R}\min_{V_1,\ldots,V_A}\max_{\tau\notin V_a}\int_{B_{K^2}}|Ef_\tau|^p\leq C(K,\epsilon)R^{p\epsilon}\|f\|_{L^q}^p,
\]where $V_1,\ldots V_A$ are $(k-1)$-planes and we have used the abbreviation $\tau\notin V_a$ to denote $\text{Angle}(G(\tau),V_a)>K^{-2}$, $a=1,\ldots,A$. For each $B_{K^2}$, fix a choice of $V_1,\ldots,V_A$ so that the minimum above is attained. Then,
\begin{equation}\label{eq: va}
\int_{B_{K^2}}|Ef|^p\lesssim K^{O(1)}\max_{\tau\notin V_a}\int_{B_{K^2}}|Ef_{\tau}|^p+\sum_{a=1}^A\int_{B_{K^2}}\big|\sum_{\tau\in V_a}Ef_\tau\big|^p,
\end{equation}
where the first term can be bounded using the $k$-broad estimate, while the second term will be handled by the cone decoupling theorem of Bourgain and Demeter, which in our notation states the following. 
\begin{theorem}[\cite{BD}]\label{thm: decoupling}
Assume $\text{supp}\hat{f}\subset N_{K^{-2}}(\mathcal{C})$, the $K^{-2}$-neighborhood of the truncated cone $\mathcal{C}\subset\mathbb{R}^n$. Then on any ball $B_{K^2}$ of radius $K^2$, for any $\delta>0$,
\[
\|f\|_{L^p(B_{K^2})}\lesssim_\delta K^{\delta}\left(\sum_{\theta\in\mathcal{P}_{K^{-2}}(\mathcal{C})}\|f_\theta\|_{L^p(W_{B_{K^2}})}^2\right)^{1/2},\quad \forall 2\leq p\leq \frac{2n}{n-2},
\]where $\mathcal{P}_{K^{-2}}(\mathcal{C})$ is a partition of $N_{K^{-2}}(\mathcal{C})$ into sectors $\theta$ of dimensions $1\times K^{-1}\times\cdots\times K^{-1}\times K^{-2}$, $f=\sum_{\theta}f_\theta$ such that $\widehat{f_\theta}=\hat{f}\chi_\theta$, and $W_{B_{K^2}}$ is a weight approximately equaling to $1$ on $B_{K^2}$ and rapidly decaying outside.
\end{theorem}
Applying this theorem to the second term followed by H\"older's inequality, for the subspaces $V_a$ as in \eqref{eq: va} one obtains
\[
\int_{B_{K^2}}\big|\sum_{\tau\in V_a}Ef_\tau\big|^p\lesssim_\delta K^\delta\max(1,K^{(k-3)(p/2-1)})\sum_{\tau\in V_a}\int W_{B_{K^2}}|Ef_\tau|^p,
\]where we have observed that the number of $\tau\in V_a$ is $\lesssim \max(1, K^{k-3})$. 

Indeed, by definition, $\tau\in V_a$ means that the angle between $\tau$ and $V_a$ is less than $K^{-2}$. By Lemma 2.2 in \cite{Harris}, $\tau \cap S^{n-1}$ lies in a $C_n K^{-1}$-- neighborhood of  $ V_a \cap \mathcal{C}\cap S^{n-1}$. 
Note that $\mathcal{C}\cap S^{n-1}$ is an $(n-2)$--dimensional sphere, which we denote  by $S^{n-2}$.  Recall that $\dim V_a =k-1$, the upper bound $K^{k-3}$ then follows from the fact that $V_a \cap S^{n-2}$ has dimension $\leq k-3$. 
This is a unique feature of the cone, which is why in the definition of the broad norm, we chose to define the angle between $\tau$ and $V_a$ to be less than $K^{-2}$, in contrast to the paraboloid case where $K^{-1}$ is used.

%


Next, summing over $B_{K^2}\subset B_R$, $a=1,\ldots,A$, one has
\[
\sum_{B_{K^2}\subset B_R}\sum_{a=1}^A\int_{B_{K^2}}\big|\sum_{\tau\in V_a}Ef_\tau\big|^p\lesssim K^\delta\max(1,K^{(k-3)(p/2-1)})\sum_\tau\int W|Ef_\tau|^p,
\]where $W:=\sum_{B_{K^2}\subset B_R} W_{B_{K^2}}$ satisfies $W\lesssim 1$ on $B_{2R}$ and rapidly decays outside $B_{2R}$. Hence, combining with the $k$-broad estimate,
\begin{equation}\label{long}
\int_{B_R}|Ef|^p\leq C(K,\epsilon)R^{p\epsilon}\|f\|_{L^q}^p+CK^\delta\max(1, K^{(k-3)(p/2-1)})\sum_\tau\int_{B_{2R}}|Ef_\tau|^p,
\end{equation}from which we are going to prove by induction on the radius that
\begin{equation}\label{linearInd}
\int_{B_R}|Ef|^p\leq \bar{C}(\epsilon)R^{p\epsilon}\|f\|_{L^q}^p.
\end{equation}

This is obviously true when $R=1$ by the trivial $L^\infty$ bound of $E$. Assume now that (\ref{linearInd}) holds for radii less than $R/2$. We apply Lorentz rescaling to handle the contribution of each $f_\tau$. To do this, we first observe that our desired estimate (\ref{linearInd}) is preserved under rotations. To see this, it is easier to work with the ``lift'' of the functions $f$ on $\mathbb{R}^{n-1}$ onto the cone. For any $f\in L^{q}(2\bar{B}^{n-1}\setminus B^{n-1})$, define $F(\xi)=f(\bar\xi)$ as a function supported on cone $\mathcal{C}$, then
\[
\|F\|_{L^q(d\sigma_{\mathcal{C}})}=\|f\|_{L^q(2\bar{B}^{n-1}\setminus B^{n-1})}
\]where $d\sigma_{\mathcal{C}}$ is the pull back of the Lebesgue measure on $\mathbb{R}^{n-1}$ under the projection $\xi\mapsto\bar\xi$, and (\ref{linearInd}) can be rephrased as
\begin{equation}\label{liinearIndLift}
\|\widehat{F d\sigma_{\mathcal{C}}}\|_{L^p(B_R)}\leq \bar{C}(\epsilon)R^{p\epsilon}\|F\|_{L^q(d\sigma_{\mathcal{C}})}^p.
\end{equation}
\begin{lemma}\label{rotation invariant}
Let $F$ be a function supported on the cone $\mathcal{C}$, and $A$ be any rotation in $\mathbb{R}^n$, then the following two inequalities are equivalent:
\begin{enumerate}
\item $\|\widehat{F d\sigma_{\mathcal{C}}}\|_{L^p(\Omega)} \leq X\|F\|_{L^q(d\sigma_{\mathcal{C}})}$
\item $\|\widehat{F(A^{-1}\cdot) d\sigma_{A(\mathcal{C})}}\|_{L^p(A(\Omega))} \leq X\|F(A^{-1}\cdot)\|_{L^q(d\sigma_{A(\mathcal{C})})}$
\end{enumerate}for any set $\Omega\subset \mathbb{R}^n$.
\end{lemma}
\begin{proof}
By change of variables, since the Jacobian of the rotation is $1$, the left hand side of both inequalities are the same, so as the right hand side. 
\end{proof}

Now we start estimating each $f_\tau$, or equivalently, its lift $F_\tau$. We slightly abuse notation by using $\tau$ to denote both the sector in $2\bar{B}^{n-1}\setminus B^{n-1}$ and its lift onto the cone at the same time. For a $\tau$ fixed, by symmetry of the cone, there is no loss of generality to assume that the central line of $\tau$ is in the first quadrant of the $(\xi_{n-1},\xi_n)$-plane. (This can be achieved through a rotation fixing the $\xi_n$-axis, mapping $\mathcal{C}$ to itself and the central line of $\tau$ into the $(\xi_{n-1},\xi_n)$-plane, combined with Lemma \ref{rotation invariant}.) Next, we want to find a rotation $A$ sending the central line of $\tau$ to be lying on the positive half of the $\xi_{n-1}$-axis. In fact, $A$ is exactly the volume conserving linear transformation mapping $\xi=(\bar{\xi},\xi_n)$ to $\omega=(\bar{\omega}, \omega_n)$ such that $\xi_{n-1}=(\omega_{n-1}-\omega_n)/\sqrt{2}$, $\xi_n=(\omega_{n-1}+\omega_n)/\sqrt{2}$ and $\xi_j=\omega_j$, $j=1\ldots, n-2$, under which the original vertical cone
\[
\mathcal{C}=\left\{\xi\in\mathbb{R}^n:\,\xi_1^2+\cdots+\xi_{n-1}^2=\xi_n^2, \,\xi_n>0, \,1\leq\xi_j\leq 2, \,\forall 1\leq j\leq n-1\right\}
\]is mapped to the tilted cone
\[
{\mathcal{T}}=\left\{\omega\in\mathbb{R}^n:\, \omega_1^2+\cdots+\omega_{n-2}^2=2\omega_{n-1}\omega_n,\, \sqrt{2}\leq\omega_{n-1}\leq 2\sqrt{2},\, 1\leq\omega_j\leq 2, \,\forall 1\leq j\leq n-2\right\}.\]By a change of variable, 
\[
\widehat{F_\tau d\sigma_{\mathcal{C}}}(x)=\widehat{G_{\tau} d\sigma_{\mathcal{T}}}(y), \quad G_\tau(\omega):=F_\tau(A^{-1}(\omega))=F_\tau(\xi), \quad y:=Ax,
\]where $d\sigma_{\mathcal{T}}$ is the pushforward of $d \sigma_{\mathcal{C}}$ under the rotation $A$.

We are now ready to apply rescaling. Introduce a new coordinate $\tilde\omega$ such that
\[
\tilde\omega_j=K\omega_j,\,\forall 1\leq j\leq n-2,\quad \tilde\omega_{n-1}=\omega_{n-1},\quad \tilde\omega_n=K^2\omega_n,
\]then one has
\[
\omega_1y_1+\cdots+\omega_{n-1}y_{n-1}+\frac{\omega_1^2+\cdots+\omega_{n-2}^2}{2\omega_{n-1}}y_{n}=\tilde\omega_1\tilde{y}_1+\cdots +\tilde\omega_{n-1}\tilde{y}_{n-1}+\frac{\tilde{\omega}_1^2+\cdots+\tilde{\omega}_{n-2}^2}{2\tilde{\omega}_{n-1}}\tilde{y}_n
\]where
\[
\tilde{y}_j=K^{-1}y_j,\,\forall 1\leq j\leq n-2,\quad \tilde{y}_{n-1}=y_{n-1},\quad \tilde{y}_n=K^{-2}y_n.
\]Observing that $\{\tilde\omega:\,\omega\in\tau\}$ is contained in a constant dilation of the tilted cone $\mathcal{T}$, say, $5\mathcal{T}$, there holds
\[
|\widehat{G_\tau d\sigma_{\mathcal{T}}}(y)|=K^{-(n-2)}|\widehat{G'_\tau d\sigma_{5\mathcal{T}}}(\tilde{y})|,
\]where $G'_\tau (\tilde\omega)$ is a function on the dilated cone $5\mathcal{T}$ such that $G'_\tau (\tilde\omega)=G_\tau (\omega)$ on the dilated $A\tau$ and $0$ elsewhere. We then apply $A^{-1}$ to rotate $5\mathcal{T}$ back to the vertical position, which leads to
\[
\widehat{G'_\tau d\sigma_{5\mathcal{T}}}(\tilde{y})=\widehat{G'_\tau(A\cdot) d \sigma_{5\mathcal{C}}}(A^{-1}\tilde{y}).
\]

We now end up with a restriction problem on $5\mathcal{C}$, while in the physical space the linear transformation has sent the ball $B_R$ into a tube of dimension $RK^{-1}\times\cdots\times RK^{-1}\times RK^{-2}\times R$. There is still an obstruction preventing us from directly using the induction assumption: this tube is not contained in a ball of radius less than $R/2$. Fortunately, this can be easily overcome by covering the tube with no more than $C_0$ balls of radius $R/C_0$, where $1<C_0 \ll K\ll R$. 
 For each ball $B_{R/C_0}$, one can also assume that it is centered at the origin, as translation in the physical space corresponds to modulation in the frequency space which doesn't change the $L^q$ norm. By the symmetry of $B_{R/C_0}$ and Lemma \ref{rotation invariant}, the induction assumption implies that
\[
\|\widehat{G'_\tau d\sigma_{5\mathcal{T}}}\|^p_{L^p(B_{R/C_0})}\leq\bar{C}(\epsilon)R^{p\epsilon}C_0^{-p\epsilon}\|G'_\tau\|^p_{L^q(d\sigma_{5\mathcal{T}})}=\bar{C}(\epsilon)R^{p\epsilon}C_0^{-p\epsilon}K^{(n-2)\frac{p}{q}}\|G_\tau\|^p_{L^q(d\sigma_{\mathcal{T}})}.
\]Then, collecting the equalities above,
\begin{equation}\label{ftau}
\begin{split}
\int_{B_{2R}}|Ef_\tau|^p&=\|\widehat{Fd\sigma_{\mathcal{C}}}\|^p_{L^p(B_{2R})}=K^{n-(n-2)p}\sum_{B_{R/C_0}}\|\widehat{G'_\tau d\sigma_{5\mathcal{T}}}\|^p_{L^p(B_{R/C_0})}\\
&\leq \bar{C}(\epsilon)R^{p\epsilon}C_0^{1-p\epsilon}K^{n-(n-2)p+(n-2)\frac{p}{q}}\|G_\tau\|^p_{L^q(d\sigma_{\mathcal{T}})}\\
&=\bar{C}(\epsilon)R^{p\epsilon}C_0^{1-p\epsilon}K^{n-(n-2)p+(n-2)\frac{p}{q}}\|f_\tau\|^p_{L^q}.
\end{split}
\end{equation}
Plugging this back into (\ref{long}), one has
\[
\begin{split}
\int_{B_R}|Ef|^p\leq &C(K,\epsilon)R^{p\epsilon}\|f\|_{L^q}^p+\\
&\qquad C\bar{C}(\epsilon)R^{p\epsilon}C_0^{1-p\epsilon}\max(1,K^{(k-3)(p/2-1)})K^{\delta+n-(n-2)p+(n-2)\frac{p}{q}}\sum_\tau\|f_\tau\|_{L^q}^p.
\end{split}
\]Observe that there are $\lesssim K^{n-2}$ sectors $\tau$ in total and recall that $p\leq q$, H\"older's inequality implies that
\[
\sum_\tau\|f_\tau\|_{L^q}^p\leq K^{(n-2)(1-\frac{p}{q})}\|f\|_{L^q}^p.
\]Plugging this into the inequality above, one has
\[
\begin{split}
\int_{B_R}|Ef|^p\leq &C(K,\epsilon)R^{p\epsilon}\|f\|_{L^q}^p+\\
&\qquad C\bar{C}(\epsilon)R^{p\epsilon}C_0^{1-p\epsilon}\max(1,K^{(k-3)(p/2-1)})K^{\delta+n-(n-2)p+(n-2)}\|f\|_{L^q}^p
\end{split}
\]where the dependence on $q$ in the exponent of $K$ cancels out.

Then the induction closes if the exponent (excluding $\delta$) of $K$ is strictly negative (so that one can always choose $\delta>0$ small enough to keep the exponent negative). Note that the presence of $C_0$ will not harm us, as for any fixed $\epsilon$, it makes negligible contribution when $K$ is sufficiently large. When $2\leq k\leq 3$, 
\[
n-(n-2)p+(n-2)<0\quad \iff\quad p>2\cdot\frac{n-1}{n-2};
\]when $k>3$,
\[
(k-3)(p/2-1)+n-(n-2)p+(n-2)<0\quad \iff\quad p>2\cdot\frac{2n-k+1}{2n-k-1}.
\]These give exactly the desired lower bound $p(k,n)$, as claimed in (\ref{linear}).

\subsection{$L^q\to L^p$ restriction} This subsection is devoted to the proof of Theorem \ref{mainpq}, again, using the $k$-broad estimate Theorem \ref{kBroadThm}. 

\subsubsection{Interior of (\ref{admpq})} When the pair $(p,q)$ lies strictly inside the interior of the claimed range in (\ref{admpq}), the estimate follows from a very similar argument as in the $L^p\to L^p$ case, so we only sketch the necessary modifications that are needed in our current case $q<p$. More precisely, fix any $R>0$, when $2\leq q\leq p\leq \frac{2n}{n-2}$, Theorem \ref{kBroadThm} tells us that
\[
\|Ef\|_{BL^p_{k,A}(B_R)}\lesssim_{k,\epsilon}R^\epsilon\|f\|_{L^q},\quad \forall p\geq \bar{p}(k,n)=2\cdot\frac{n+k}{n+k-2}.
\]We are going to show that there holds
\begin{equation}\label{linearpq}
\|Ef\|_{L^p(B_R)}\lesssim_\epsilon R^\epsilon \|f\|_{L^q}
\end{equation}whenever $2\leq p\leq \frac{2n}{n-2}$ and
\begin{equation}\label{pqk}
\begin{cases}
p\geq 2\cdot\frac{n+2}{n},\,q'< \frac{n-2}{n}p, &\text{if}\,\, k=2,\\
p\geq \bar{p}(k,n),\, p>\frac{n}{\frac{2n-k-1}{2}-\frac{n-k+1}{q}}, &\text{if}\,\,k\geq 3,
\end{cases}
\end{equation}for some $2\leq k\leq n$.

As in the previous subsection, we start with estimating the ``broad'' part of $Ef$ by Theorem \ref{kBroadThm} and treating the ``narrow'' part of it using the decoupling theorem of Bourgain--Demeter. After decoupling, we apply H\"older's inequality to change from $\ell^2$ to $\ell^q$, reaching the estimate
\[
\int_{B_{K^2}}\big|\sum_{\tau\in V_a}Ef_\tau\big|^p\lesssim_\delta K^\delta\max(1,K^{(k-3)(\frac{1}{2}-\frac{1}{q})p})\left(\sum_{\tau\in V_a}\big(\int W_{B_{K^2}}|Ef_\tau|^p\big)^{\frac{q}{p}}\right)^{\frac{p}{q}}.
\]Summing over $a=1\ldots,A$ and then $B_{K^2}\subset B_R$ using Minkowski inequality,
\[
\begin{split}
&\sum_{B_{K^2}\subset B_R}\sum_{a=1}^A\int_{B_{K^2}}\big|\sum_{\tau\in V_a}Ef_\tau\big|^p\\
\lesssim& K^\delta\max(1,K^{(k-3)(\frac{1}{2}-\frac{1}{q})p})\sum_{B_{K^2}\subset B_R}\left(\sum_{\tau}\big(\int W_{B_{K^2}}|Ef_\tau|^p\big)^{\frac{q}{p}}\right)^{\frac{p}{q}}\\
\lesssim& K^\delta\max(1,K^{(k-3)(\frac{1}{2}-\frac{1}{q})p})\left(\sum_\tau\Big(\sum_{B_{K^2}\subset B_R}\big(\int W_{B_{K^2}}|Ef_\tau|^p\big)^{\frac{q}{p}\cdot\frac{p}{q}}\Big)^{\frac{q}{p}}\right)^{\frac{p}{q}}\\
\lesssim& K^\delta\max(1,K^{(k-3)(\frac{1}{2}-\frac{1}{q})p})\left(\sum_\tau\big(\int_{B_{2R}}|Ef_\tau|^p\big)^{\frac{q}{p}}\right)^{\frac{p}{q}}+\text{RapDec}(R)\|f\|_{L^q}^p,
\end{split}
\]where we have summed up $W_{B_{K^2}}$ to a single weight $W$ as in the previous subsection. This gives us a slightly different form of (\ref{long}):
\begin{equation}\label{longpq}
\int_{B_R}|Ef|^p\leq C(K,\epsilon)R^{p\epsilon}\|f\|_{L^q}^p+CK^\delta\max(1,K^{(k-3)(\frac{1}{2}-\frac{1}{q})p})\left(\sum_\tau\big(\int_{B_{2R}}|Ef_\tau|^p\big)^{\frac{q}{p}}\right)^{\frac{p}{q}}.
\end{equation}

We then proceed in the exact same way as in the previous subsection to apply induction on scales to get (\ref{ftau}). Without the need of using H\"older's inequality, one can plug it into (\ref{longpq}) to directly obtain
\[
\begin{split}
\int_{B_R}|Ef|^p\leq &C(K,\epsilon)R^{p\epsilon}\|f\|_{L^q}^p+\\
&\quad\quad C\bar{C}(\epsilon)R^{p\epsilon}C_0^{1-p\epsilon}\max(1,K^{(k-3)(\frac{1}{2}-\frac{1}{q})p})K^{\delta+n-(n-2)p+(n-2)\frac{p}{q}}\|f\|_{L^q}^p.
\end{split}
\]The induction closes if the exponent (excluding $\delta$) of $K$ is strictly negative. When $k=2$, 
\[
n-(n-2)p+(n-2)\frac{p}{q}<0\quad \iff\quad q'<\frac{n-2}{n}p; 
\]when $k\geq 3$,
\[
(k-3)(\frac{1}{2}-\frac{1}{q})p+n-(n-2)p+(n-2)\frac{p}{q}<0\quad \iff\quad p>\frac{n}{\frac{2n-k-1}{2}-\frac{n-k+1}{q}}.
\]These are exactly the desired conditions in (\ref{pqk}).

\begin{remark}In the case $q=2$, the range of tuples $(p,q,k)$ in (\ref{pqk}) is empty for all $2\leq k\leq n$, which explains the extra restriction one needs to put on $q$ in the admissibility condition (\ref{admpq}). Moreover, the elimination of the endpoint of the range of $p$ follows from $\epsilon$-removal.
\end{remark}

\subsubsection{Boundary of (\ref{admpq})} In the previous subsection, we have already obtained the desired linear restriction estimate for all pairs $(p,q)$ that lie strictly inside the claimed range (\ref{admpq}), it is thus left to examine the boundary case $q'=\frac{n-2}{n}p$ when $k=2$ and $p=\frac{n}{\frac{2n-k-1}{2}-\frac{n-k+1}{q}}$ when $k>3$. 

In order to do this, we apply a bilinear interpolation adapted from Theorem 2.2 of \cite{TVV} where the case of the paraboloid is studied. The key idea here is that linear and bilinear restriction estimates are essentially equivalent on the boundary line of (\ref{admpq}).
\begin{theorem}
Let $n\geq 3$ and $1<p,q<\infty$ be such that $2p>\frac{2(n-1)}{n-2}$ and $q'\leq \frac{n-2}{n}\cdot 2p$. Let $R(q\to 2p)$ denote the linear cone restriction estimate
\[
\|Ef\|_{L^{2p}(\mathbb{R}^n)}\lesssim \|f\|_{L^q}
\]and $R(q\times q\to p)$ denote the bilinear cone restriction estimate
\[
\|(Ef_1)(Ef_2)\|_{L^p(\mathbb{R}^n)}\lesssim \|f_1\|_{L^q}\|f_2\|_{L^q}
\]for all functions $f_i$ supported in $U_i\subset 2B^{n-1}\setminus B^{n-1}$ such that $U_1,U_2$ are transversal. Then, 
\begin{enumerate}
\item $R(q\to 2p)$ implies $R(q\times q\to p)$;
\item if $R(\tilde{q}\times \tilde{q}\to \tilde{p})$ holds for all $(\tilde{p},\tilde{q})$ in a neighborhood of $(p,q)$, then $R(q\to 2p)$ holds.
\end{enumerate}
\end{theorem}
It seems that this theorem has not been explicitly stated in the literature before, but the proof of which is very similar to Theorem 2.2 of \cite{TVV}. In particular, the direction of linear implying bilinear restriction simply follows from H\"older's inequality. In order to conclude linear restriction from the bilinear restriction, one partitions the cone into sectors at different scales and explore the quasi-orthogonality between pairs of sectors that are close to each other at each scale, which follows from the bilinear restriction estimate after applying Lorentz rescaling as in Subsection \ref{prescale} above. This then yields enough decay for all the scales to be summable. In fact, when $n\geq 4$, the proof proceeds exactly the same as in Theorem 2.2 of \cite{TVV} after replacing $n$ by $n-1$. When $n=3$, one needs to work through the argument separately as the case $n-1=2$ is not covered in their theorem, but there is no new difficulty that arises. We omit the details. 

Therefore, given $2\leq k\leq n$ and a point $(p,q)$ on the boundary of the region (\ref{admpq}), it suffices to find a neighborhood of $(p,q)$ where the bilinear restriction holds true. Such a neighborhood can be found by interpolating the bilinear restriction in the interior of (\ref{admpq}) that is implied by the linear estimate, together with the following theorem of Wolff.
\begin{theorem}[Theorem 1 of \cite{Wolff}]
For $n\geq 3$ and $p>1+\frac{2}{n}$, the bilinear cone restriction estimate $R(2\times 2\to p)$ holds true.
\end{theorem}

\section*{Acknowledgments} 
The authors would like to thank Larry Guth for suggesting the problem and multiple enlightening discussions, as well as for carefully reading through a first draft of the article. The authors are also indebted to Ciprian Demeter, Terence Tao and Ana Vargas for helpful conversations on Lorentz rescaling and bilinear interpolation. The first author is supported in part by NSF-DMS \#1854148. This material is based upon work supported by the National Science Foundation under Grant No. DMS-1440140 while the authors were in residence at the Mathematical Sciences Research Institute in Berkeley, California, during the Spring 2017 semester.

%



\bibliography{RestrictionCone}{}
\bibliographystyle{amsplain}
\end{document}